\documentclass[11pt]{article}
\usepackage{amsfonts}
\usepackage{latexsym}
\usepackage{graphicx}
\usepackage{subcaption}
\usepackage{color, amsmath,amssymb, amsfonts, amstext,amsthm, latexsym, dsfont}
\usepackage{adjustbox}  % 导言区添加
\usepackage{epic,eepic}
\usepackage{graphicx}
\usepackage{mathrsfs}
\usepackage{epstopdf}
\usepackage{longtable}
\usepackage{booktabs}
\usepackage{amsmath, amssymb}
\usepackage[title]{appendix}
\usepackage{array}
\usepackage{indentfirst}
\usepackage{graphicx}
\usepackage{float}
\usepackage{graphicx}
\usepackage[utf8]{inputenc}
\usepackage{url}
\usepackage{booktabs}
\usepackage[colorlinks=false, linktocpage=true]{hyperref}
\usepackage{amsmath, amssymb}
\usepackage{amsmath, amssymb}
\usepackage{algorithm}
\usepackage{algpseudocode}  % 这是新的替代版，支持行号和自定义
\usepackage{caption}
\allowdisplaybreaks[4]

\usepackage{tikz}
\usepackage{etoolbox}
\usepackage{enumitem}
\usepackage{soul}
\soulregister{\cite}7 % 注册\cite命令
\soulregister{\citep}7 % 注册\citep命令
\soulregister{\citet}7 % 注册\citet命令
\soulregister{\ref}7 % 注册\ref命令
\soulregister{\pageref}7 % 注册\pageref
\usepackage{xcolor}

\definecolor{awesome}{rgb}{1.0, 0.13, 0.32}
\definecolor{bronze}{rgb}{0.8, 0.5, 0.2}
\definecolor{burntorange}{rgb}{0.8, 0.33, 0.0}
\definecolor{blue(ncs)}{rgb}{0.0, 0.53, 0.74}
\theoremstyle{plain}
\newtheorem{theorem}{Theorem}[section]
\newtheorem{lemma}[theorem]{Lemma}

\theoremstyle{remark}
\newtheorem{remark}[theorem]{Remark}

\theoremstyle{definition}
\newtheorem{definition}[theorem]{Definition}
\newtheorem{assumption}{Assumption}
\numberwithin{equation}{section} % 设置equation的编号格式

\begin{document}

\title{\bf  Critical Transitions in Interacting Particle Systems: An Onsager–Machlup Action Functional Framework}

%%% optional title：

%% Mean Field Control of Interacting Particle Systems with Tipping Transitions: An Onsager–Machlup Approach

%% OM-GAN: A Generative Framework for Controlling Tipping Phenomena in Interacting Particle Systems

%% Tipping and Critical Transitions in Multi-Agent Dynamics: A Generative Mean_Field Control Framework

%% Controlling Tipping Dynamics in Interacting Particle Systems via Onsager–Machlup Functionals

\author{\bf\normalsize{
Jianyu Chen$^{1,}$\footnotemark[2],
Ting Gao $^{1,}$\footnotemark[1],
Galina Strelkova $^{2,}$\footnotemark[3]
Jinqiao Duan$^{1,3,4,}$\footnotemark[4]
}\\[10pt]
\footnotesize{$^1$ School of Mathematics and Statistics, Huazhong University of Science and Technology, Wuhan 430074, China.} \\[5pt]
\footnotesize{$^2$ Department of Nonlinear Dynamics and Electromagnetism, Saratov State University, Russia.}\\[5pt]
\footnotesize{$^3$ Department of Mathematics and Department of Physics, Great Bay University, Dongguan, Guangdong 523000, China.}\\[5pt]
\footnotesize{$^4$ Dongguan Key Laboratory for Data Science and Intelligent Medicine, Dongguan, Guangdong 523000, China.} 
}

\footnotetext[2]{Email: \texttt{jianyuchen@hust.edu.cn}}
\footnotetext[1]{Email: 
\texttt{tgao0716@hust.edu.cn}}
\footnotetext[5]{Email: \texttt{duan@gbu.edu.cn}}
\footnotetext[3]{Email: \texttt{strelkovagi@sgu.ru}}
\footnotetext[1]{Corresponding author}

\date{\today}
\maketitle
\vspace{-0.3in}
\begin{abstract}
This paper establishes an indirect approximation theorem for the most probable transition pathway of a stochastic interacting particle system in the mean-field framework. This paper studied the problem of indirect approximation of the most probable transition pathway of an interacting particle system (i.e., a high-dimensional stochastic dynamical system) and its mean field limit equation (McKean-Vlasov stochastic differential equation). This study is based on the Onsager-Machlup action functional, reformulated the problem as an optimal control problem.  This paper completes the derivation using the stochastic Pontryagin's Maximum Principle. 
This paper proves the existence and uniqueness theorem for the solution to the mean-field optimal control problem of McKean-Vlasov stochastic differential equations and establishes a system of equations that determine the control parameters $\theta^{*}$ and $\theta^{N}$ respectively. There are few studies on the most probable transition pathways of stochastic interacting particle systems, it is still a great challenge to solve the most probable transition pathways directly or to approximate it with the mean field limit system. Therefore, this paper first gave the proof of correspondence between the core equation of Pontryagin's Maximum Principle, that is, Hamiltonian extreme condition equation. In other words, this relationship indirectly illustrates the correspondence between the most probable transition pathways of stochastic interacting particle systems and those of mean-field systems.

         \textbf{Keywords and Phrases:} Most probable transition pathway, Optimal control, Pontryagin's Maximum Principle, McKean-Vlasov Stochastic Differential Equation, Generative Adversarial Network (GAN) .

\end{abstract}
%----------------------------------------Section 1---------------------------------------------
\section{ Introduction}
In recent years, particle systems with interactions have been extensively studied from various perspectives including mathematics, physics, chemistry, and biology \cite{liggett1985interacting}. This field has attracted significant interest from many researchers. Interacting particle systems are composed of multiple microscopic particles that interact with each other, representing an area of interdisciplinary research across various fields. Currently, research in this area is highly active, with a wide range of theoretical explorations and diverse applications.  In statistical physics, theoretical investigations of interacting particle systems primarily focus on simulating and understanding the behavior of complex systems such as phase transitions, critical phenomena \cite{konno1995phase}, and others. In the field of chemistry, research on interacting particle systems mainly involves intermolecular interactions, chemical reaction kinetics \cite{mielke2017non}, catalyst design, and related aspects. In biology, the study of interacting particle systems mainly deals with interactions among biomolecules, intracellular signaling, protein folding, assembly, etc. For instance, game theory-based models of cancer cell and stroma cell dynamics utilize interacting particle systems \cite{zheng2020game}.

The mean-field limit equations of interacting particle systems provide an approximate approach for studying many-body systems \cite{golse2016dynamics}. This method typically assumes that each particle is influenced by the average effect of all other particles, without accounting for specific pairwise interactions. Such an approximation is often justified and helps simplify the analysis of complex many-body systems. The mean-field limit equations generally consist of a set of differential equations that describe the evolution of each particle in the system \cite{yong2017linear}."

These equations are typically grounded in the principles of dynamics and statistical physics to capture the macroscopic behavior of the system. They describe the evolution of macroscopic properties such as particle positions, velocities, and momenta, while specific interactions between particles are modeled through the influence of a mean field. The exact form of the mean-field limit equations depends on the characteristics of the system being studied and the nature of the interactions involved. For instance, in statistical physics, mean-field equations can effectively describe the collective behavior of a large number of particles. Overall, mean-field limit equations offer a powerful approximation method that simplifies the analysis of many-body systems and extends to the study of stochastic dynamical systems.

Mean-field limit equations for particle systems can take various forms, with one prominent example being the McKean-Vlasov equation. The McKean-Vlasov equation is a type of partial differential equation that describes the dynamics of many-body systems by modeling the evolution of the particle density function over time while accounting for the effects of particle interactions. Specifically, the general form of the McKean-Vlasov equation is given by the follows:
$$
\frac{\partial f}{\partial t} + v \cdot \nabla_x f + F(t,x,f,\nabla_x f) \cdot \nabla_v f = 0.
$$
Here, \(f(t,x,v)\) represents the density function of particles, describing the density of particles at time \(t\), position \(x\), and velocity \(v\). And \(F(t,x,f,\nabla_x f)\) is a given function representing the dependence of the particle density function \(f\) on position \(x\) and velocity \(v\), typically depending on the specific characteristics and interaction forms of the system.

The McKean-Vlasov equation is widely used to describe the macroscopic behavior of systems with a large number of particles, such as gases and fluids. It serves as a powerful mathematical tool for analyzing many-body systems and holds substantial theoretical importance in understanding the dynamics of complex systems.

\subsection{Control Theory of McKean-Vlasov Type Stochastic Differential Equations}

The analysis of stochastic differential equations (SDEs) of McKean-Vlasov type has a long history. These equations were initially introduced by McKean \cite{mckean1966class} with the aim of rigorously treating certain nonlinear partial differential equations (PDEs). Subsequently, scholars delved into the study to address problems in their respective fields and to extend them to broader contexts. Jourdain et al. \cite{jourdain2007nonlinear} explored the existence and uniqueness of solutions to McKean-Vlasov type stochastic differential equations. The properties of solutions were discussed within the framework of chaotic propagation theory, as McKean-Vlasov equations appear to be effective models for describing large-scale particle dynamics influenced by mean-field interactions.

However, the optimal control problem driven by McKean-Vlasov stochastic differential equations (SDEs) seems to be a relatively new topic, with limited research in the literature related to stochastic control. The stochastic control problem for McKean-Vlasov stochastic differential equations shares many similarities with the mean-field game problem initially proposed by Lasry and Lions \cite{lasry2007mean}, as well as concurrently by Caines, Huang, and Malham\'{e} \cite{mei2020closed}. The differences and similarities between these two problems were analyzed and discussed in Carmona et al. \cite{carmona2013control}, emphasizing that solving the mean-field game problem involves optimizing before searching for fixed points, whereas by searching for fixed points before optimizing, we can obtain solutions to the optimal control problem of McKean-Vlasov SDEs.

\subsection{Pontryagin's Maximum Principle for The Most Probable Transition Pathways}
Looking from another perspective, optimal control theory is capable of transforming variational problems into corresponding optimal control problems without requiring numerical solution of the Euler-Lagrange equations \cite{wei2022optimal}. Optimal control theory naturally arises alongside variational methods. There are two interrelated approaches for detecting optimal controls: Pontryagin's Maximum Principle (PMP) and the Hamilton-Jacobi-Bellman (HJB) principle. A fascinating historical account of the development of these theories can be found in the literature \cite{liberzon2011calculus}. Optimal control problems are also closely related to dynamical systems, and one of the main methods for solving optimal control problems is to derive a set of necessary conditions (i.e., the Euler-Lagrange differential equations). These conditions must be satisfied by any optimal trajectory solution. Optimal control problems can be viewed as optimization problems in infinite-dimensional spaces, thus they are often challenging to solve. Although sufficient and necessary conditions for first and second-order optimization exist \cite{liberzon2011calculus}, they still pose significant challenges for numerical computations.

Pontryagin and his team proposed and derived the Maximum Principle in the 1950s, marking a true milestone in optimal control theory. It states that any optimal control problem with an optimal trajectory solution must address what is known as the extended Hamiltonian system \cite{liberzon2011calculus}. Similarly, this also involves a two-point boundary value problem (also known as forward-backward differential equations), along with a maximization condition on the Hamiltonian function. The mathematical significance of Pontryagin's Maximum Principle lies in making the maximization of the Hamiltonian much easier than the original infinite-dimensional control problem. This enables the derivation of closed-form solutions for certain types of optimal control problems, including the case of linear-quadratic systems. The Maximum Principle has demonstrated its applicability across various disciplines. For instance, Bartholomew-Biggs optimized spacecraft orbits using Pontryagin's Maximum Principle \cite{bartholomew1980optimisation}.

\subsection{Onsager-Machlup Action Functional and Maximum Probability Transition Trajectories}

As various stochastic factors are considered, stochastic dynamical systems have become effective tools for studying complex phenomena. They are widely applied in modeling various fields such as physics \cite{arnold1995random,duan2015introduction}, biology \cite{wang2007effective,imkeller2001stochastic}, and finance \cite{rosenthal2012mathematical}. Stochastic differential equations, as mathematical models, are prevalent across different domains including physics \cite{temam2012infinite}, biology \cite{kreutz2009systems}, engineering, and finance \cite{rosenthal2012mathematical}. They account for stochastic fluctuations due to environmental factors, making them important models for simulating complex phenomena and predicting rare events \cite{gao2016dynamical}. The stochastic fluctuations in these systems can lead to unexpected rare events. Under the influence of external noise, the dynamical behavior described by stochastic differential equations can differ significantly from deterministic differential equations \cite{herr2023three}. For deterministic differential equations, state transitions between metastable states under the vector field do not occur. However, even with minor noise influence, state transitions between equilibrium states of the vector field described by stochastic differential equations may occur. Literature on stochastic differential equations mainly focuses on Gaussian dynamics, i.e., stochastic differential equations under Brownian motion \cite{duan2015introduction}, which has found applications in various fields. Biswas et al. \cite{biswas2021characterising} focused on numerically characterizing the volume of attractors in the state space of dynamical systems excited by additive Gaussian white noise. Other studies have investigated the behavior between states of prosperity and extinction in population systems influenced by delayed and correlated Gaussian colored noise, as well as the phenomenon of stochastic resonance \cite{wang2021impact}. 

From classical Newtonian mechanics, we know that as long as the initial state of a system and the laws governing the change of system state parameters over time are known, the state of the system at any time can be predicted. In reality, phenomena in engineering and natural sciences are inevitably subject to noise interference. These complex noise sources may arise from interactions among various units within the system, external random disturbances, random initial conditions, and so on. Therefore, noise becomes the most common stochastic factor. Dynamical systems also exhibit a high response to noise, thereby demonstrating various dynamics driven by noise, including noise-induced transitions \cite{horsthemke1984noise,kitahara1979coloured}, stochastic resonance \cite{mori2002noise}, chaos \cite{lai2011transient}, and state transitions \cite{jourdain2007nonlinear}. Noise-induced migration phenomena occur in various systems, such as chemical reactions \cite{dykman1994large} and physically dynamic switching systems \cite{bomze2012noise}. This interesting migration phenomenon often arises due to the appearance of noise, which alters the deterministic dynamical behavior of the original deterministic system, causing stable states in the system to be disturbed and becoming metastable. The dynamic properties of metastable states in the system are unstable, leading to the occurrence of state transitions \cite{budhiraja2019analysis}. This class of unstable system's stochastic fluctuations may trigger rare events, and studying such migration phenomena can help us understand the nature of dynamical systems more intuitively. For example, the properties of migration trajectories and quantifying the impact of stochastic noise on dynamical systems can help understand the essence of abrupt changes in complex systems. For many irreversible systems, the absence of equilibrium states makes it difficult to analyze their asymptotic behavior and migration phenomena.

The Freidlin-Wentzell large deviation theory and Onsager-Machlup action functional theory are effective tools for studying such migration phenomena. However, the Freidlin-Wentzell large deviation theory focuses on perturbations with infinite time and infinitesimal noise. The Onsager-Machlup action functional theory characterizes the most probable transition paths of diffusion processes with nonzero noise and can effectively solve the problem of state transitions in stochastic dynamical systems driven by noise of certain intensity within a finite time. Therefore, we adopt the Onsager-Machlup action functional theory to study migration phenomena within a finite time. For example, the change in substance concentration after a certain reaction time in chemical reaction systems \cite{maier1993escape}, the change in carbon dioxide concentration over time in the carbon cycle system, and the change in the population of biological species over time in river aquatic plant systems. The significance of the Onsager-Machlup action functional theory lies in our concern for the state transition problem within a certain migration time T, which is more practically significant for predicting the occurrence of rare events and controlling major natural disasters. In addition, the Onsager-Machlup action functional has been applied in data assimilation \cite{sugiura2017onsager}, fluctuation theorems \cite{singh2008onsager}, and quantum physics \cite{onsager1953fluctuations}, among other fields. The Onsager-Machlup action functional can be used to study the most probable migration trajectory of stochastic dynamical systems because it quantifies the probability of sample trajectories in the neighborhood of any reference trajectory within a tubular region. By means of the Onsager-Machlup action functional, we can obtain the probability distribution of solution trajectories of stochastic dynamical systems, thereby calculating the most probable migration trajectory. The Onsager-Machlup action functional measures the probability of rare events, such as the maximum probability transition trajectory between metastable states. Under the constraint of connecting two metastable states, the extremum (usually expressed as a minimum value) of the action functional is considered the most probable migration path. Therefore, from the perspective of this functional, the most probable migration trajectory is the trajectory with the maximum probability, which corresponds to the minimum value point of the Onsager-Machlup action functional. Thus, we have explained the significance and solution approach of the most probable migration trajectory of stochastic dynamical systems. In summary, the problem of the most probable migration trajectory of stochastic dynamical systems can be regarded as a minimization problem of the Onsager-Machlup action functional.

Onsager and Machlup \cite{onsager1953fluctuations} were the first to study the distribution of sample trajectories of a class of diffusion processes, focusing on the probability within a given neighborhood. Subsequently, Stratonovich et al. \cite{feller1954diffusion} extensively studied the Onsager-Machlup action functional theory in stochastic differential equations and provided rigorous mathematical derivations. The key to the derivation lies in the Girsanov transformation, which transforms the transition probability of the diffusion

% \section{Related Work}\label{sec2}

\section{Preliminaries}\label{sec3}
In this section, we prepare for the main theorems to be deduced later. This chapter mainly introduces the commonly used mathematical symbols and important mathematical assumptions, lemmas and so on. Firstly, a kind of Brown type random interacting particle system studied in this chapter and its corresponding mean field limit equation McKean-Vlasov random differential equation are introduced in detail, and then the important reference theorems are introduced in detail. The Onsager-Machlup functional of McKean-Vlasov stochastic differential equation is included. Finally, the research of control theory on McKean-Vlasov stochastic differential equations is introduced, especially the necessary and sufficient conditions of solutions.

\subsection{Stochastic interacting particle system and its mean field McKean-Vlasov stochastic differential equation}

There is a connection between studying detailed descriptions of the laws of particle evolution and simplified descriptions, and this connection is usually established through the mean field theory. The mean field theory allows us to consider the collective behavior of a large number of particles from a system, rather than the behavior of each particle individually. By averaging the interactions between particles, the equation describing the whole behavior of the system can be obtained.

For example, on the one hand there is the Liouville equation \cite{burkholder1991topics} :
\begin{equation}
\partial_t u+\sum_1^N v_i \partial_{x_i} u+\sum_{j \neq i}-\nabla V_N\left(x_i-x_j\right) \partial_{v_j} u=0,
\end{equation}
Where $x_i$represents the position of the particle, $v_i$represents the velocity of the particle, and $u\left(t, x_1, v_1, \ldots, x_n, v_N\right)$is the existence density at time $t$, assuming that the interaction function $v$is symmetric with respect to $N$particles. Call $V_N(\cdot)$a potential interaction of pairs. On the other hand, there is the following Boltzmann equation \cite{burkholder1991topics}
\begin{equation}
\partial_t u+v \cdot \nabla_x u=\int_{\boldsymbol{R}^3 \times S_2}\left(u(x, \tilde{v}) u\left(x,  \tilde{v}^{\prime}\right)-u(x, v) u\left(x,  v^{\prime}\right)\right)\left|\left(v^{\prime}-v\right) \cdot n\right| d v^{\prime} d n,
\end{equation}
Where $\tilde{v}, \tilde{v}^{\prime}$is obtained by exchanging the corresponding components of $v$and $v^{\prime}$in the direction of $n$, i.e. :
$$
\begin{aligned}
& \tilde{v}=v+\left(v^{\prime}-v\right) \cdot n n ,\\
& \tilde{v}^{\prime}=v^{\prime}+\left(v-v^{\prime}\right) \cdot n n.
\end{aligned}
$$
Here, $u(t, x, v)$is the location of $x$, the velocity of $v$, and the time of existence density of $t$.

\begin{definition} \label{Definition 5.1}(\textbf{Stochastic Differential Equation for Stochastical Interacting particle System}) \quad For $\rm N$particles on $\mathbb{R}^d$, Assuming its initial distribution is $u_0^{\otimes N}$, the stochastic differential equation (SDE) satisfies the following form:
\begin{equation}
d x_t^i=\sigma d B_t^i+\frac{1}{N} \sum_1^N b\left(x_t^i, x_t^j\right) d t, \quad 1 \leq i \leq N.
\end{equation}
Where $B^i$ is the independent identically distributed Brown motion in $\mathbb{R}^d$, $b$ is the drift coefficient of $\mathbb{R}^d \times \mathbb{R}^d \to \mathbb{R}^d$, $\sigma$ is the diffusion coefficient, Represents the noise intensity, here taken as a constant.
\end{definition}

For the following nonlinear equation
\begin{equation}
\begin{aligned}
& \partial_t u=\frac{1}{2} \Delta u-\operatorname{div}\left(\int b(\cdot, y) u(t, y) d y u\right), \\
& u_{t=0}=u_0.
\end{aligned}
\end{equation}
The % is generated during the propagation of the chaos effect.
Let $P_t^0$ represent the Brown transition density, and according to the perturbation formula we have:
\begin{equation}
u_t(x)-u_0 P_t^0(x)=\int_0^t d s_1 \int d x_1 d x_2 u_{s_1}\left(x_1\right) u_{s_1}\left(x_2\right) b\left( x_1,  x_2\right) \nabla_{x_1} P_{t-s_1}^0\left(x_1, x\right) .
\end{equation}
Continuing the same perturbation for $u_{s_1}\left(x_1\right) u_{s_1}\left(x_2\right), \ldots$, we find by induction:
\begin{equation}
\begin{aligned}
& u_t=u_0 P_t^0+\sum_{k=1}^m \int_{0<s_k<\ldots<s_1<t} d s_k  d s_1 u_0^{\otimes k+1} P_{s_k}^0 B \cdot \nabla P_{s_{k-1}-s_k}^0  B \cdot \nabla P_{t-s_1}^0+R_m, \\
& R_m=\int_{0<s_{m+1}<s_m<\cdots<s_1<t} d s_{m+1}  d s_1 u_{s_{m+1}}^{\otimes m+2} B \cdot \nabla P_{s_m-s_{m+1}}^0  \nabla P_{t-s_1}^0 .
\end{aligned}
\end{equation}
Where $P_t^0$ acts as a tensor on a function of any number of independent variables, $B \cdot \nabla$maps a function of $k$ to a variable of $(k+ 1)$ in the following way:
$$
[B \cdot \nabla] f\left(x_1, \ldots, x_{k+1}\right)=\sum_1^k b\left(x_i, x_{i+1}\right) \nabla_i f\left( x_1, \ldots,  x_k\right) .
$$

\begin{definition} (\textbf{McKean Diffusion Process}) \quad Suppose the function $b:\mathbb {R}^d \times \mathbb{R}^d \rightarrow \mathbb{R}^d$ Satisfies Lipschitz boundedness, And in the $\left(\mathbb{R}^d \times C_0\left( \mathbb{R}_{+}, \mathbb{R}^d\right)\right)^{\boldsymbol{N}^*}$ space have product metric $\left(u_0 \otimes W\right)^ {\otimes N^*} $ ($u_0 $ is probability in $\mathbb{R}^d$, $B $ $\in \mathbb {R} ^ d $ is standard Brownian motion), particle $X ^ {I, N}, I = 1, \ldots, N $, satisfied
\begin{equation}
\begin{aligned}
& d X_t^{i, N}=d B_t^i+\frac{1}{N} \sum_{j=1}^N b\left(X_t^{i, N}, X_t^{j, N} \right) d t, \quad i=1, \ldots, N ,\\
& X_0^{i, N}=x_0^i.
\end{aligned}
\end{equation}
Here $x_0^i,\left(w^i\right), i \geq 1$ are the canonical coordinates on the product space $\left(R^d \times C_0\right)^{N^*}$.
\end{definition}

As the number of particles $\rm N$approaches infinity, each particle $X^{i, N}$corresponds to a natural limit $\bar{X}^i$. Each $\bar{X}^i$pair corresponds to a new nonlinear process, which we call McKean-Vlasov Stochastic Differential Equation (SDE) \cite{burkholder1991topics}.

\begin{definition} ( \textbf{Brown type McKean-Vlasov SDE} ) \quad  Suppose there is a probability space $\left(\Omega, F, F_t,\left(B_t\right)_{t \geq 0}, X_0, P\right)$, equipped with $\mathbb{R}^{d} -valued$ Brownian motion $\left(B_t\right)_{t \geq 0}$. $X_0$ is $F_0$ measurable, and has distribution $\mu_0$. The stochastic process $X_t$ satisfies the following McKean-Vlasov stochastic differential equation (SDE) :
\begin{equation}\label{MVSDE}
\begin{aligned}
& d X_t=\int b\left(X_t, y\right) \mu_t(d y) d t + \sigma d B_t, \quad 0 \leq t \leq T, \\
& X_{t=0}=X_0.
\end{aligned}
\end{equation}
Here $\mu_t$ is the distribution of $X_t$, and $\sigma$ represents the intensity of Brownian noise.
\end{definition}

\begin{remark}
(i) At present, there are two main categories of noise disturbance terms in McKean-Vlasov stochastic differential equations: one is Gaussian noise, which is simple but very applicable. In mathematics, it is the generalized time derivative of Brown motion, which is an important kind of stationary Gaussian process with some good properties, such as continuous sample orbit and light tail of probability density function. The other is non-Gaussian noise, which is mainly simulated by L\'{e}vy process. It is a very important random process, which has different properties from Brown motion, mainly reflected in its discontinuous sample orbit and heavy tail of probability density function.

(ii) Since this paper studies the migration orbit problem of stochastic dynamical systems based on Onager -Machlup functional theory of action, and in this chapter, we hope to establish the Onager -Machlup functional approximation theorem of McKean-Vlasov stochastic differential equations and interacting particle systems. It should be noted that the Onsager-Machlup functional theory of action for stochastic dynamical systems driven by l\'{e}vy noise is not mature even in additive cases. For this reason, we consider stochastic interacting particle systems driven by Gaussian Brown noise and Brown type McKean-Vlasov stochastic differential equations.

According to the form of definition \ref{Definition 5.1}, our condition for the noise of a particle system is the independent uniformly distributed Brown noise $B_t^{i}$, thus we obtain the McKean-Vlasov stochastic differential equation in the shape of the equation \eqref{MVSDE}. If we consider that the particle system is subject to the same Brown noise $B_t$, then we get a McKean-Vlasov stochastic partial differential equation. At present, the study of Onsager-Machlup action functional theory for McKean-Vlasov stochastic partial differential equations has not produced good results, although corresponding results have been obtained in the sense of large deviation.
\end{remark}

\subsection{Onsager-Machlup action functional of McKean-Vlasov stochastic differential equation}

The Onsager-Machlup action functional of classical stochastic differential equations driven by Brown motion has been studied extensively in the last few decades. Ikeda and Watanabe \cite{ikeda2014stochastic} derive the Onager -Machlup action functional for the reference path $\phi \in C^2([0, 1], \mathbb{R}^d)$in the highest norm sense. Shepp and Zeitouni \cite{shepp1992note} show that this result holds for the highest norm equivalent in Cameron-Martin Spaces. Liu et al. \cite{liu2023onsager} In a recent work in 2023, the Onsager-Machlup action functional of a special class of McKean-Vlasov stochastic differential equations with drift function $f$is derived. Next, we give the basic definitions and symbols of the mathematical quantities needed in this chapter.

Let $\mathscr{P}$ is the space of all the probabolity measures $\mu \in \mathbb{R}^{\mathrm{d}}$, and let
$$
\mathscr{P}_2\left(\mathbb{R}^{\mathrm{d}}\right)=\left\{\mu \in \mathscr{P}\left(\mathbb{R}^{ d}\right): \mu\left(|\cdot|^2\right):=\int_{\mathbb{R}^{\mathrm{d}}}|x|^2 \mu( \mathrm {d}x)<\infty\right\}.
$$
Here $\mathscr{P}_2$ is the p-complete, separable metric space under the Wasserstein metric. Next define the Wasserstein metric.

\begin{definition}(\textbf{Coupling of Probability Measures}) \quad let $\mu$ and $\nu$ be the probability metric space $\mathscr{P}_2\left(\mathbb{R}^{\mathrm{d}}\right)$ Suppose that for $\pi \in \mathscr{C}(\mu, v) \in \mathbb{R}^{\mathrm{d}} \times \mathbb{R}^{d}$, the following two conditions are true: \\
(i) $\pi\left(\cdot \times \mathbb{R}^{\mathrm{d}}\right)=\mu$; \\
(ii)$\pi\left(\mathbb{R}^ {\mathrm{d}} \times \cdot\right)=v$.
Are we call set $\mathscr{C}(\mu, v) \in \mathbb{R}^{\mathrm{d}} \times \mathbb{R}^{d}$ for the probability measure decoupling collection.
\end{definition}

\begin{definition} (\textbf{Wasserstein Metric} \cite{liu2023onsager}) \quad Let $\mu$ and $\nu$ are two probability measure in the probability  metric space $\mathscr{P}_2\left(\mathbb{R}^{\mathrm{d}}\right)$. Wasserstein Metric of $\mu$ and $\nu$ are:
\begin{equation}
\mathbb{W}_2(\mu, v):=\inf _{\pi \in \mathscr{C}(\mu, v)}\left(\int_{\mathbb{R}^{\mathrm{ d}} \times \mathbb{R}^{\mathrm{d}}}|x-y|^2 \pi(\mathrm{d} x, \mathrm{~d} y)\right)^{\frac{ 1}{2}}, \mu, v \in \mathscr{P}_2\left(\mathbb{R}^{\mathrm{d}}\right).
\end{equation}
\end{definition}

\begin{remark}
(i) For any random variable $X$ and $Y$ with the value in $\mathbb{R}^d$, we have
$$
\mathbb{W}_2\left(\mathscr{L}_X, \mathscr{L}_Y\right) \leq\left[E|X-Y|^2\right]^{\frac{1}{2}} ,
$$
Where $\mathscr {L} _ {\ xi} $said random variable $\ xi $in $\ \mathbb {R} ^ d $on distribution.

(ii) If $\phi_t$is a definite track, the distribution of the track $\phi_t$is called the Dirac measure, i.e. $\mathscr{L}_{\phi_t}=\delta_{\phi_t} $.
\end{remark}

\begin{definition} \cite{ren2020space} Let $T \in(0, \infty]$, when the time $T=\infty$, $[0, T]=[0, \infty)$.

(i) For functions $h: \in \mathscr{P}_2\left(\mathbb{R}^{\mathrm{d}}\right)$, if functional
$$
L^2\left(\mathbb{R}^{\mathrm{d}} \rightarrow \mathbb{R}^{\mathrm{d}}, \mu\right) \ni \phi \mapsto h\left(\mu \circ(\operatorname{Id}+\phi)^{-1}\right)
$$
is Fr\'{e}chet differentiable at $\phi=0 \in L^2\left(\mathbb{R}^{\mathrm{d}} \rightarrow \mathbb{R}^{\mathrm{d}}, \mu\right)$. That is to say, existence (unique) $\xi \in L^2\left(\mathbb{R}^{\mathrm{d}} \rightarrow \mathbb{R}^{\mathrm{d}}, \mu\right)$ such that
$$
\lim _{\mu\left(|\phi|^2\right) \rightarrow 0} \frac{h\left(\mu \circ(\mathrm{Id}+\phi)^{-1}\right)-h(\mu)-\mu(\langle\xi, \phi\rangle)}{\sqrt{\mu\left(|\phi|^2\right)}}=0 .
$$
Then we call the function $h: \mathscr{P}_2\left(\mathbb{R}^{\mathrm{d}}\right) \rightarrow \mathbb{R}$in$\mu \in \mathscr{P}_2\left(\mathbb{R}^{\mathrm{d}}\right)$is $L$- differentiable. Note $\partial_\mu h(\mu)=\xi$, and it is  called the $L$- derivative of the function $h$at $\mu$.
 
(ii) If for all $\mu$ in $\mathscr{P}_2\left(\mathbb{R}^{\mathrm{d}}\right)$, function $h: \mathscr{P}_2\left(\mathbb{R}^{\mathrm{d}}\right) \rightarrow \mathbb{R}$ has L- derivative $\partial_\mu h(\mu)$, then $h$ is L-differentiable in $\mathscr{P}_2\left(\mathbb{R}^{\mathrm{d}}\right)$. In addition, if$\left(\partial_\mu h(\mu)\right)(y)$  a $y \in \mathbb{R}^{\mathrm{d}}$ on differentiable version, And $\left(\partial_\mu h(\mu)\right)(y)$ and $\partial_y\left(\partial_\mu h(\mu)\right)(y)$ in $(\mu, y) \in \mathscr{P}_2\left(\mathbb{R}^{\mathrm{d}}\right) \times \mathbb{R}^{\mathrm{d}}$ is continuous. We write $h \in C^{(1,1)}\left(\mathscr{P}_2\left(\mathbb{R}^{\mathrm{d}}\right)\right)$.

(iii) If for all parameters in $[0, T] \times \mathbb{R}^{\mathrm{d}} \times \mathscr{P}_2\left(\mathbb{R}^{\mathrm{d}}\right)$, function $h:[0, T] \times \mathbb{R}^{\mathrm{d}} \times \mathscr{P}_2\left(\mathbb{R}^{\mathrm{d}}\right) \rightarrow \mathbb{R}$ derivative $\partial_t h(t, x, \mu), \partial_x h(t, x, \mu), \partial_x^2 h(t, x, \mu), \partial_\mu h(t, x, \mu)(y), \partial_y \partial_\mu h(t, x, \mu)(y)$ exists. And it's joint continuous in $(t, x, \mu)$ or $(t, x, \mu, y)$.  The function is said to belong to the class $C^{1,2,(1,1)}$. If all the derivatives in $[0, T] \times \mathbb{R}^{\mathrm{d}} \times \mathscr{P}_2\left(\mathbb{R}^{\mathrm{d}}\right)$ are bounded, then we remark the function $f$ belongs to the class $f \in C_b^{1,2,(1,1)}$.

(iv) If function $h \in C^{1,2,(1,1)}\left([0, T] \times \mathbb{R}^{\mathrm{d}} \times \mathscr{P}_2\left(\mathbb{R}^{\mathrm{d}}\right)\right)$ and  $$
(t, x, \mu) \mapsto \int_{\mathbb{R}^{\mathrm{d}}}\left\{\left\|\partial_y \partial_\mu h\right\|+\left\|\partial_\mu h\right\|^2\right\}(t, x, \mu)(y) \mu(\mathrm{d} y)
$$ are locally bounded. That is to say, they are bounded in a compact subset of $[0, T] \times \mathbb{R}^{\mathrm{d}} \times \mathscr{P}_2\left(\mathbb{R}^{\mathrm{d}}\right)$. Then we have $$
h \in \mathscr{C}\left([0, \infty) \times \mathbb{R}^{\mathrm{d}} \times \mathscr{P}_2\left(\mathbb{R}^{\mathrm{d}}\right)\right).
$$
\end{definition}

With the above basic definition and mathematical notation, we can derive the Onsager-Machlup functional theorem of McKean-Vlasov stochastic differential equation.

\begin{theorem} \label{Theorem 5.1}\cite{liu2023onsager} \textbf{(McKean-Vlasov SDE's Onager -Machlup action functional)}\quad Consider the following McKean-Vlasov stochastic differential equation:
$$
\mathrm{d} X_t=f\left(t, X_t, \mathscr{L}_{X_t}\right) \mathrm{d} t+\mathrm{d} B_t, X(0)=x_0,
$$
Where $f:[0,T] \times \mathbb{R}^{\mathrm{d}} \times \mathscr{P}_2\left(\mathbb{R}^{\mathrm{d}}\right) \rightarrow \mathbb{R}^{\mathrm{d}}, B_t$ is Brownian motion in $\mathbb{R}^d$.   Given a complete probability space $\left(\Omega, \mathcal{F},\left(\mathcal{F}_t\right)_{t>0}, \mathbb{P}\right)$, $\mathscr{L}_{X_t}$ is the distribution of $X_t$. Assuming that the conditions (H1), (H2) and (H3) are all satisfied, $X_t$is the solution of a random differential equation, and the reference path $\phi$is the function that makes $\phi_t-x$belong to the Cameron-Martin space $\mathcal{H}$, and assume the drift function $f \in C_b^{1,2,(1,1)}\left([0,1] \times \mathbb{R}^{\mathrm{d}} \times \mathscr{P}_2\left(\mathbb{R}^{\mathrm{d}}\right)\right)$. Then for any $L^2\left([0,1], \mathbb{R}^{\mathrm{d}}\right)$norm, the Onager-Machlup functional of $X_t$ exists and has the following form:
$$
L\left(t, \phi, \dot{\phi}, \delta_\phi\right)=\int_0^T\left|\dot{\phi}_t-f\left(t, \phi_t, \mathscr{L}_{\phi_t}\right)\right|^2 \mathrm{~d} t+\int_0^T \operatorname{div}_x f\left(t, \phi_t, \mathscr{L}_{\phi_t}\right) \mathrm{d} t.
$$
Here $\operatorname{div}_x f=\sum_{i=1}^{\mathrm{d}} \partial_{x_i} f_i\left(t, \phi_t, \mathscr{L}_{\phi_t}\right)$  represents the divergence of $\phi_t \in \mathbb{R}^{\mathrm{d}}$.
\end{theorem}

\section{Study on the Most Probable Transition Pathway of Stochastic Interacting Particle Systems based on the Stochastic Pontryagin's Maximum Principle}
In this section, we mainly aim to establish the approximation theorem of the maximum possible migration orbit for randomly interacting particle systems. In the background of the research, we state the mean field approximation theorem, so it is meaningful to consider the maximum possible transfer orbit of the mean field limit McKean-Vlasov stochastic differential equation to approximate the orbit of the particle system. Next, we establish the mean field approximation theorem for the maximum possible migration orbit of the particle system.

\subsection{The Most Probable Transition Pathway for Stochastic Interacting Particle Systems}
In general, the stochastic interacting particle systems in $\mathbb{R}^d$ are described by the following stochastic differential equations:
\begin{equation} \label{paticle system}
d X_t^i=f^i\left(t, X_t^1, \ldots, X_t^N\right) d t+\sigma^i\left(t, X_t^1, \ldots, X_t^N\right) d B_t^i, \quad i=1,\cdots, N,
\end{equation}
where $f^i:[0, T] \times \mathbb {R}^{Nd}\to\mathbb{R}^d$ are the drift functions, $\sigma^i:[0, T] \times \mathbb {R}^{Nd}\to\mathbb{R}^d\times\mathbb{R}^k$ are the diffusion functions, and $B^i$ are independent $\mathbb{R}^k$-valued Brownian motions on the filtered probability space $(\Omega,\mathcal{F},\mathbb{F}=\left(\mathcal{F}_t\right)_{0 \leq t \leq T},\mathbb{P})$. 

In our framework, we shall consider special cases of stochastic interacting particle systems \eqref{paticle system}. More precisely, let $k=d$, the drift function $f^i(t,x^1,\cdots,x^N)=\frac{1}{N} \sum_{j=1}^N b(x^i, x^j)$ for some function $b:\mathbb{R}^d\times\mathbb{R}^d\to\mathbb{R}^d$ and the diffusion functions $\sigma^i$ be constants denoted by $\sigma$. To this end, we consider the following stochastic interacting particle systems
\begin{equation} \label{paticle system1}
d Y^N_t=F(Y^N_t) d t+\sigma d W_t^N, 
\end{equation}
where $Y^N_t = (X_t^1,\cdots,X_t^N)^{\top}$, $W^N_t = (B_t^1,\cdots,B_t^N)^{\top}$ and 
$$
F(Y^N_t) = \left(\frac{1}{N} \sum_{j=1}^N b(X_t^{1}, X_t^j),\cdots, \frac{1}{N} \sum_{j=1}^N b(X_t^{N}, X_t^j)\right)^{\top}.
$$  
The Onsager-Machlup action functional corresponding to the stochastic interacting particle systems \eqref{paticle system1} is:
\begin{equation} \label{5.16}
S^N_{T}\left(x, \dot{x}\right)=\frac{1}{2}\int_0^TL^N(x(t),\dot{x}(t))
\mathrm{~d} t,
\end{equation}
where $L^N$ is the Lagrangian of the Onsager-Machlup action functional \eqref{5.16} expressed as 
\begin{equation}\label{lagrangian}
L^{N}(x, \dot{x})=\left \|\sigma^{-2} 
  (\dot{x}-F\left( x\right))\right\|^2+ \operatorname{div} F\left(x\right)  .
\end{equation}

The stochastic differential equation \eqref{paticle system1} driven by high Vega Brown noise satisfies the following optimal control problems: 
\begin{equation}\label{5.19}
\begin{cases}
\underset{\theta \in \Theta }{\operatorname{min}} & J^N(y,\dot{y})=\int_0^T L^{N}( y(t), \dot{y}(t)) \mathrm{d} t+\Phi\left(y(T)\right), \\ 
\text{ s.t } & \dot{y}(t)=F(y(t)) +\sigma \theta(t).
 \end{cases}
\end{equation}
Here $\Theta$ is the set of all $\mathbb{R}^{Nd}$-valued control functions and the terminal cost function $\Phi$ is a real valued deterministic function. 
According to definition \ref{Definition 5.9}, the above deterministic optimal control problem satisfies the Pontryagin maximum principle. The most important equation in Pontryagin's maximum principle is the maximum condition. The optimality condition of the optimal control $\theta^{N}$corresponding to the optimal control problem \eqref{5.19} should satisfy the following formula:
\begin{equation}
\boldsymbol{Q}_N\left(\boldsymbol{\theta}^N\right)_t:=\frac{1}{N} \sum_{i=1}^N \nabla_\theta H\left(x_t^{\boldsymbol{\theta}^N, i}, p_t^{\boldsymbol{\theta}^N, i}, \theta_t^N\right)=0
\end{equation}
We call $\boldsymbol{Q}_N\left(\boldsymbol{\theta}^N\right)_t$the equation satisfied by the optimal control $\theta^{N}$in a random interacting particle system \eqref{paticle system}.

The numerical algorithm for solving the most probable transition pathway of a particle system is difficult due to the limitation of dimensionality. Next, we hope to establish the correspondence between the maximum possible migration orbit of a particle system and its corresponding average field limit system.

\subsection{The Most Probable Transition Pathway  for Stochastic Dynamical Systems with Mean Field Limit McKean-Vlasov Equation}
Assume that $B=\left(B_t\right)_{0 \leq t\leq T}$ is the standard $k$ dimensional Brown motion defined on the probability space $(\Omega, \mathcal{F}, \mathbb{P})$, $\mathbb{F}=\left(\mathcal{F}_t\right)_{0 \leq t \leq T}$ is its natural $\sigma $- algebra. For each random variable/vector or random process $X$, we use $\mu_X$ to represent the distribution of $X$, Use $\mathscr{P}_2\left(\mathbb{R}^{\mathrm{d}}\right)$to represent the complete, divisible metric space under the Wasserstein metric, and have $\mu \in \mathscr{P}_2$.

According to the work of Sznitman \cite{burkholder1991topics}, considering the additive Brown noise drive, the McKean-Vlasov stochastic differential equation satisfies the following form:
\begin{equation} \label{5.10}
\begin{aligned}
& d X_t=\int b\left(X_t, y\right) \mu_t(d y) d t + \sigma d B_t,  \quad 0 \leq t \leq T,\\
& X_{t=0}=X_0.
\end{aligned}
\end{equation} 
Here $\mu_t$ is the distribution of $X_t$, and $\sigma$represents the Brown noise intensity.

\begin{assumption}
The function $b: \mathbb{R}^d \times \mathbb{R}^d \to \mathbb{R}^d \to \mathbb{r}^d$ has a decomposition such that $b(X_t, y) = h(X_t) y$ holds.
\end{assumption}

Assuming the above assumptions are met, we have $\int b\left(X_t, y\right) \mu_t(d y) = \int h\left(X_t\right) y \mu_t(d y)$. For the constant $p \geq 1$, there is $h\left(X_t\right)\left(\int y \mu_t(d y)\right)^p=h\left(X_t\right)\left[E X_t\right]^p$, in particular, We take $p =1$, and further we can write the equation \eqref{5.10} as:
\begin{equation} \label{5.11}
\begin{aligned}
& d X_t=h\left(X_t\right)\left[E (X_t)\right] d t + \sigma d B_t,  \quad 0 \leq t \leq T,\\
& X_{t=0}=X_0.
\end{aligned}
\end{equation} 
Let $f(t, X_t) = h(t, X_t) [E(X_t)]$, where $f:[0,T] \times \mathbb{R}^{\mathrm{d}} \times \mathscr{P}_2\left(\mathbb{R}^{\mathrm{d}}\right) \rightarrow \mathbb{R}^{\mathrm{d}}$ is diffusion coefficient. Next, according to theorem \ref{Theorem 5.1}, and the above assumptions are true, we have the following statement.

Suppose in the system \eqref{5.11}, drift function $f(t, X_t) = h(t, X_t) [E(X_t)]$ satisfy assumptions  \textbf{H1-H3} in the paper \cite{liu2023onsager}, specifically $f \in C_b^{1, 2,(1,1)}\left([0,T] \times \mathbb{R}^{\mathrm{d}} \times \mathscr{P}_2\left(\mathbb{R}^{\mathrm{d}}\right)\right)$. Then for the $L^2\left([0,1], \mathbb{R}^{\mathrm{d}}\right)$ norm, the Onager-Machlup functional is:
\begin{equation} \label{MVSDEOM}
S^{OM}_{T}\left(t, \phi, \dot{\phi}\right)=\frac{1}{2}\int_0^T \left[\rm {B}^{-1}
\left|\dot{\phi}_t-f\left(t, \phi_t, \mu_{\phi_t}\right)\right|^2 + \operatorname{div}_x f\left(t, \phi_t,  \mu_{\phi_t}\right)  \right]
\mathrm{~d} t.
\end{equation}
Here $\rm B = \sigma \sigma^{*}$represents the diffusion coefficient matrix. Moreover, we define a Lagrangian as:
$$L\left(t, \phi, \dot{\phi}\right) = \frac{1}{2} \left[\rm {B}^{-1} 
  \left|\dot{\phi}_t-f\left(t, \phi_t, \mu_{\phi_t}\right)\right|^2\right]+\frac{1}{2} \operatorname{div}_x f\left(t, \phi_t, \mu_{\phi_t}\right),
  $$
and represents the divergence of  $\phi_t \in \mathbb{R}^{\mathrm{d}}$ as 
$$\operatorname{div}_x f=\sum_{i=1}^{\mathrm{d}} \partial_{x_i} f_i\left(t, \phi_t, \mu_{\phi_t}\right).
$$

Next, in order to better establish the maximum possible migration orbit approximation theorem for random interacting particle systems in the sense of mean field, we need to restate the equation \eqref{5.11} and the equation \eqref{MVSDEOM} as an optimal control problem. That is, it is assumed that there is a control $\theta \in \Theta$, which makes the following mean field optimal control problem valid.

\begin{equation}\label{5.13}
    \begin{cases}
    \underset{\theta \in \Theta }{\operatorname{min}} & \mathbb{E}_{\mu_0} \left[\frac{1}{2}\int_{0}^{T} [\theta^{2}+\operatorname{div}_{X}f(t, X_t, \mu_{X_t})]dt+g(X(T), \mu_{X(T)})\right], \quad  \mu_0 \in \mathscr{P}_2\left(\mathbb{R}^{\mathrm{d}}\right),\\ 
    \text{ s.t } & \dot{ X_t}=h\left(X_t\right)\left[E (X_t)\right] d t + \sigma \theta,  \quad 0 \leq t \leq T, \\
    &  {X}(0)=x_0, \quad {X}(T)=x_{T}.
    \end{cases}
\end{equation}

Next we need to introduce concepts and notations related to stochastic optimal control theory.

\begin{assumption}\label{Hypothesis 5.2} For the McKean-Vlasov stochastic differential equation (SDE) \eqref{5.11}, the following two assumptions are satisfied:

(A1) The function $t \in [0, T] \mapsto(f, \sigma) \in \mathbb{R}^d \times \mathbb{R}^ {d \times k}$is square-integrable;
%\left(t, 0, \delta_0, 0\right)

(A2)  $\exists c>0,  \forall t \in[0, T],  \forall x, x^{\prime} \in \mathbb{R}^d, \forall \mu,  \mu^{\prime} \in \mathcal{P}_2\left(\mathbb{R}^d\right)$, has the following formula:
$$
\begin{aligned}
& \left|f(t, x, \mu)-f\left(t, x^{\prime}, \mu^{\prime}\right)\right|+\left|\sigma(t, x, \mu)-\sigma\left(t, x^{\prime}, \mu^{\prime}\right)\right| \\
& \quad \leq c\left[\left|x-x^{\prime}\right|+\mathbb{W}_2\left(\mu, \mu^{\prime}\right)\right].
\end{aligned}
$$
Here $\mathbb{W}_2\left(\mu, \mu^{\prime}\right)$represents the 2-Wasserstein distance.

\end{assumption}

\begin{remark}
When $p>1$, $p$-Wasserstein distance $\mathbb{W}_p\left(\mu, \mu^{\prime}\right)$ on $\mathcal{P}_p(E)$ are defined as follows:
$$
\begin{aligned} & \mathbb{W}_p\left(\mu, \mu^{\prime}\right)=\inf \left\{\left[\int_{\mathbb{E} \times \mathbb{E} }|x-y|^p \pi(d x,  d y)\right]^{1 / p} ; \right. \\ &\left.\pi \in \mathcal{P}_2(\mathbb{E}  \times \mathbb{E} ),   \mu   \text {and}  \mu^{\prime}  \text {are marginal probability measures}\right\}.\end{aligned}
$$
\end{remark}

\begin{theorem} \label{theorem 5.3} \textbf{(\textbf{Existence and Uniqueness Theorem of Equation \eqref{5.11} })} \quad Let $\theta \in \Theta$ is the control, and $\mathbb{A}$ is the set of  all controlled process$X^{\theta}_t$,  where $X^{\theta}_t \in \mathbb{H}^{2, k}$, $\mathbb{H}^{2, d}$ is the Hilbert space:
$$
\mathbb{H}^{2, d}:=\left\{Z \in \mathbb{H}^{0, d} ; \mathbb{E} \int_0^T\left|Z_s\right|^2 d s<+\infty\right\}.
$$
Here $\mathbb{H}^{0, d}$represents the set of all $\mathbb{R}^d$valued sequentially measurable processes on $[0, T]$. According to the hypothesis \ref{Hypothesis 5.2} (A1) and (A2), any $X^{\theta}_t \in \mathbb{A}$is satisfied
$$
\mathbb{E} \int_0^T\left[\left|f\left(t, X^{\theta}_t, \mu_t \right)\right|^2+\left|\sigma\left(t, X^{\theta}_t, \mu_t \right)\right|^2\right] d t<+\infty .
$$
Combined with the Lipschitz hypothesis (A2), this guarantees that for any controllable process $X^{\theta}_t \in \mathbb{A}$, there exists a unique solution to the equation \eqref{5.11} $X^{\theta^*}_t$, And this solution also satisfies that for every $p \in[1,2]$, there is
$$
\mathbb{E} \sup _{0 \leq t \leq T}\left|X^{\theta^*}_t\right|^p<+\infty.
$$
\end{theorem}

\begin{proof}
About the proof, please see, such as literature \cite {jourdain2007nonlinear, burkholder1991topics}.
\end{proof}

Next, we consider the problem of the maximum possible migration orbit corresponding to the equation \eqref{5.11}. According to the Onager -Machlup action functional \eqref{MVSDEOM}, Bringing the Lagrange $L (t, \phi, \dot{\phi})$into the mean field optimal control problem \eqref{5.13}, there is the following minimization problem:
\begin{equation}\label{5.14}
\underset{\theta \in \Theta }{\operatorname{min}} \quad J(\theta)=\mathbb{E}_{\mu_0}\left\{\int_0^T L\left(t, \phi, \dot{\phi}, \theta \right) d t+g\left(X_T, \mu_{X_T}\right)\right\}.
\end{equation}
Here the running cost function $L$is a real valued deterministic function defined on $[0, T] \times \mathbb{R}^d\times \mathcal{P}_2\left(\mathbb{R}^d\right)$, The terminal cost function $g$is also a real valued deterministic function, defined on $\mathbb{R}^d \times \mathcal{P}_2\left(\mathbb{R}^d\right)$.

\subsection{Stochastic Pontryagin's Maximum Principle for McKean-Vlasov Stochastic Differential Equations}
\begin{definition} \label{Definition 5.7}(\textbf{Joint Differentiability} \cite{a2cd2729-9bcf-3536-aa3b-80b81a878489}) \quad Consider the function  $h: \mathbb{R}^d \times \mathcal{P}_2\left(\mathbb{R}^d\right) \ni(x, \mu) \rightarrow h(x, \mu) \in \mathbb{R}$. If lift function $\tilde{h}: \mathbb{R}^d \times L^2\left(\tilde{\Omega} ; \mathbb{R}^d\right) \ni(x, \tilde{X}) \mapsto h\left(x, \tilde{\mathbb{P}}_{\tilde{X}}\right)$ is jointly differentiable, then $h$is jointly differentiable. Next define the partial derivatives of $x$ and $\mu$ as:
$$
\mathbb{R}^d \times \mathcal{P}_2\left(\mathbb{R}^d\right) \ni (x, \mu) \mapsto \partial_x h(x, \mu),
$$
$$
\mathbb{R}^d \times \mathcal{P}_2\left(\mathbb{R}^d\right) \ni(x, \mu) \mapsto \partial_\mu h(x, \mu)(\cdot) \in L^2\left(\mathbb{R}^d, \mu\right).
$$
Thus, the partial derivative of the function $\tilde{h}$in the direction $\tilde{X}$ Fr\'{e}chet is
$$
L^2\left(\tilde{\Omega} ; \mathbb{R}^d\right) \ni(x, \tilde{X}) \mapsto D_{\tilde{X}} \tilde{h}(x, \tilde{X})=\partial_\mu h(x, \tilde{\mathbb{P}} \tilde{X})(\tilde{X}) \in L^2\left(\tilde{\Omega} ; \mathbb{R}^d\right).
$$    
\end{definition}
\begin{remark}
We often use the fact that joint continuous differentiability in two parameters corresponds to the joint continuity of the divergentibility of each of the two parameters and the partial derivatives. Here, the joint continuity of $\partial_x h$is understood as the joint continuity with respect to the Euclidean distance on $\mathbb{R}^d$and the Wasserstein distance on $\mathcal{P}_2\left(\mathbb{R}^d\right)$.
 The joint continuity of $\partial_\mu h$is understood to be the mapping from $\mathbb{R}^n \times L^2\left(\tilde{\Omega} ; \mathbb{R}^d\right)$  to $L^2\left(\tilde{\Omega} ; \mathbb{R}^d\right)$. That is the joint continuity of $(x, \tilde{X}) \mapsto \partial_\mu h\left(x, \tilde{\mu}_{\tilde{X}}\right)(\tilde{X})$.
 \end{remark}

 \begin{definition} (\textbf{Convex Function of Measure} \cite{a2cd2729-9bcf-3536-aa3b-80b81a878489}) \quad For a differentiable function $h$that satisfies the definition \ref{Definition 5.7}, If for all$\mu \in \mathcal{P}_2\left(\mathbb{R}^d\right)$ and $\mu^{\prime} \in \mathcal{P}_2\left(\mathbb{R}^d\right)$, we have:
 $$
h\left(\mu^{\prime}\right)-h(\mu)-\tilde{\mathbb{E}}\left[\partial_\mu h(\mu)(\tilde{X}) \cdot\left(\tilde{X}^{\prime}-\tilde{X}\right)\right] \geq 0
$$
Here $\tilde{X}$and $\tilde{X}^{\prime}$are square-integrable random variables with distributions $\mu$and $\mu^{\prime}$, then the function $h$is said to be convex.
\end{definition}
\begin{remark}
More generally, for functions $h$ that are jointly differentiable in the above sense, $\tilde{X}$ and $\tilde{X}^{\prime}$ are square integrable random variables with distributions $\mu$ and $\mu^{\prime}$. If for each  $(x, \mu) \in \mathbb{R}^n \times$ $\mathcal{P}_2\left(\mathbb{R}^d\right)$ and $\left(x^{\prime}, \mu^{\prime}\right) \in \mathbb{R}^n \times$ $\mathcal{P}_2\left(\mathbb{R}^d\right)$, we have 
$$
\begin{gathered}
h\left(x^{\prime}, \mu^{\prime}\right)-h(x, \mu)-\partial_x h(x, \mu) \cdot\left(x^{\prime}-x\right) \\
-\tilde{\mathbb{E}}\left[\partial_\mu h(x, \mu)(\tilde{X}) \cdot\left(\tilde{X}^{\prime}-\tilde{X}\right)\right] \geq 0.
\end{gathered}
$$
Then the function $h$ is said to be convex.
\end{remark}

\begin{definition}\label{Definition 5.9} (\textbf{Hamiltonian of random Pontryagin maximum principle}\cite{a2cd2729-9bcf-3536-aa3b-80b81a878489}) \quad The Hamiltonian of random Pontryagin maximum principle is defined as a function $H$, as follows:
$$
H(t, x, \mu, p, z, \theta)=f(t, x, \mu) \cdot p - L(t, x, \dot{x}, \theta).
$$
Here the dot symbol represents the inner product in Euclidean space. Since we need to compute the derivative of $H$with respect to its variable $\mu$, we consider the raised Hamiltonian $\tilde{H}$defined in the following way:
$$
\tilde{H}(t, x, \tilde{X}, p, \theta)=H(t, x, \mu, p, \theta).
$$
Here $\tilde{X}$is any random variable with a distribution of $\mu$, and we will use $\partial_\mu H(t, x, \mu_0, p, \theta)$to represent the derivative calculated against $\mu$at $\mu_0$, Where all other variables $t, x, p$and $\theta$remain the same.
\end{definition}

\begin{remark}
Here we emphasize that $\partial_\mu H\left(t, x, \mu_0, p, \theta\right)$is an element of $L^2\left(\mathbb{R}^d, \mu_0\right)$, We associate this with the function $\partial_\mu H\left(t, x, \mu_0, p, \theta\right)(\cdot): \mathbb{R}^d \ni \tilde{x} \mapsto \partial_\mu H\left(t, x, \mu_0, p, \theta\right)(\tilde{x})$is equated. It meets the following conditions:
$$
D \tilde{H}(t, x, \tilde{X}, p, \theta)=\partial_\mu H\left(t, x, \mu_0, p, \theta\right)(\tilde{X}),
$$
The above conditions hold almost everywhere in the sense of the measure $\tilde{\mu}$.
\end{remark}

\begin{definition} \label{Definition 5.10} (\textbf{Conjugate Equations of Stochastic Optimal Pontryagin Maximum Principle}) \quad For McKean-Vlasov Stochastic Differential Equation (SDE) \eqref{5.11} The drift function $f$and the diffusion coefficient $\sigma$satisfy the assumption \ref{Hypothesis 5.2} (A1)-(A2), and the assumption coefficients $f, \sigma$and the terminal cost function $g$in the equation \eqref{5.13} are jointly differentiable for $x$and $\mu$.  Then, given an acceptable control $\theta=\left(\theta_t\right)_{0 \leq t\ leq T} \in \Theta$, we denote the corresponding controlled state process by $X=X^\theta$. When the following conditions are satisfied:
$$
\begin{aligned}
\mathbb{E} \int_0^T\Big\{\Big|\partial_x f\big(t, X_t, \mu_{X_t}, \theta_t\big)\Big|^2&+\tilde{\mathbb{E}}\Big[\Big|\partial_\mu f\big(t, X_t, \mu_{X_t}, \theta_t\big)\big(\tilde{X}_t\big)\Big|^2\Big]\Big\} d t<+\infty, \\
\mathbb{E}\Big\{\Big|\partial_x g\big(X_T, \mu_{X_T}\big)\Big|^2&+\tilde{\mathbb{E}}\Big[\Big|\partial_\mu g\big(X_T, \mu_{X_T}\big)\big(\tilde{X}_T\big)\Big|^2\Big]\Big\}<+\infty,
\end{aligned}
$$
The conjugation process of $P_t$, which we call $X_t$, satisfies the following equations (which we call conjugate equations) :
$$
\begin{aligned}
d P_t= & -\partial_x H\left(t, X_t, \mu_{X_t}, P_t, \theta_t\right) d t \\
& -\tilde{\mathbb{E}}\left[\partial_\mu H\left(t, \tilde{X}_t, \mu_{X_t}, \tilde{P}_t,  \tilde{\theta}_t\right)\left(X_t\right)\right] d t. \\
\end{aligned}
$$
The $(\tilde{X}, \tilde{P},  \tilde{\theta})$  defined on  $L^2(\tilde{\Omega}, \tilde{\mathcal{F}}, \tilde{\mu})$  independent replication of ${X}, {P},  \theta)$, $\mathbb{E}$ is the expectations on $(\tilde{\Omega}, \tilde{\mathcal{F}}, \tilde{\mu})$.
\end{definition}

\section{Main conclusions and proofs}

With the important definitions and mathematical notation of the previous section, the main conclusions of this chapter are stated in this section. The first is about McKean-Vlasov randomness
The differential equation (SDE) \eqref{5.11} corresponds to the Pontryagin maximum principle for stochastic optimal control problems. Secondly, we derive the optimal control functional and approximation theorem of McKean-Vlasov SDE's Pontryagin maximum principle for random interacting particle systems.
\begin{theorem}\label{theorem 5.4} (\textbf{Existence and uniqueness of solutions to the mean field optimal control problem}) \quad In the mean field optimal control problem \eqref{5.13}, for the Hamiltonian $H$of the form \ref{Definition 5.9}, If we further assume that the Hamiltonian $H$is convex for the control $\theta$, we can accept that the control $\left(\theta^{*}_t\right)_{0 \leq t\ leq T} \in \Theta$is optimal, $\left(X_t^{\theta}\right)_{0 \leq t \leq T}$is the relevant optimal controlled state, $\left(P_t\right)_{0 \leq t\ leq T}$is the associated conjugation process satisfying the definition \ref{Definition 5.10}, then we have:
\begin{equation}
\begin{aligned}
& \forall \theta \in \Theta \quad H\left(t, X^{\theta}_t, \mu_{X_t}, P_t, \theta_t\right) \leq H\left(t, X_t, \mu_{X_t}, P_t,  \theta^{*}\right), \\
& d t \otimes d \mu \text {a.s.}
\end{aligned}
\end{equation}
\end{theorem}

\begin{proof}
Since the control set $\Theta$is convex, given another control $\theta_1 \in \Theta$, We can choose to perturb $\theta^{\varepsilon}=\theta^*_t+\varepsilon\left(\theta_1-\theta^*_t\right)$, for $0 \leq \varepsilon\ leq 1$, It still belongs to $\Theta$. Since $\theta^*_t$is optimal, we have inequalities
$$
\left.\frac{d}{d \varepsilon} J(\theta_t^*+\varepsilon(\theta_1-\theta_t^*))\right|_{\varepsilon=0}=\mathbb{E} \int_0^T\left[\partial_{\theta_t} H\left(t, X_t, \mu_{X_t}, P_t,  \theta^{*}_t\right) \cdot\left(\theta_1-\theta^*_t\right)\right] d t \geq 0 .
$$
Since the Hamiltonian $H$is convex with respect to the control variable $\theta \in \Theta$, for all $\theta_1 \in \Theta$we conclude:
$$
\mathbb{E} \int_0^T\left[H\left(t, X_t, \mu_{X_t}, P_t, \theta_1\right)-H\left(t, X_t, \mu_{X_t}, P_t, \theta^{*}_t\right)\right] d t \geq 0,
$$
Now, if for a definite control $\theta_t^{*} \in \Theta$, we select $\theta_1$as follows:
$$
\theta_1(\omega)= \begin{cases}\theta_t^{*}, & \text { if }(t, \omega) \in C, \\ \theta_t(\omega), \end{cases}
$$
For any progressive measurable set $C \subset[0, T]\times \Omega$, % (i.e., for any $\left.t \in[0, T]\right)$, Have $C \ cap [0, t] \ in $$\ mathcal {B} ([0, t]) \ otimes \ mathcal {} F _t $),
$$
\mathbb{E} \int_0^T \mathbf{1}_C\left[H\left(t, X_t, \mu_{X_t}, P_t,  \theta^{*}_t\right)-H\left(t, X_t, \mu_{X_t}, P_t,  \theta_t\right)\right] d t \geq 0,
$$
then 
$$
H\left(t, X_t, \mu_{X_t}, P_t, \theta_t^{*}\right)-H\left(t, X_t, \mu_{X_t}, P_t,  \theta_t\right) \geq 0, \quad d t \otimes d \mu \text {-a.s.}
$$
At this point, we have completed the proof of the theorem that for the average field optimal control problem \eqref{5.13}, the solution of the random Pontryagin maximum principle exists and is unique. in other words, there is a unique optimal control $\theta_t^{*} \in \Theta$that minimizes the cost function $J$. 
\end{proof}

With the above theorems, the robustness of the mean field optimal control problem is guaranteed. Carmona et al. \cite{a2cd2729-9bcf-3536-aa3b-80b81a878489} derived the Pontryagin maximum condition for stochastic optimal control problems. We have the following lemma.

\begin{lemma} (\cite{a2cd2729-9bcf-3536-aa3b-80b81a878489}) Assuming that the other conditions remain consistent with theorem \ref{theorem 5.4}, but does not require the control set $\Theta$to be convex, nor does it require $H$to be convex with respect to $\theta$. Assuming that the acceptable control $\left(\theta^{*}_t\right)_{0 \leq t\ leq T} \in \Theta$is optimal, $\left(X^{\theta}_t\right)_{0 \leq t\ leq T}$is the related optimal controlled state, and $\left(P_t\right)_{0 \leq t\ leq T}$is the related conjugate process. So we have
$$
\nabla_\theta H\left(t, X_t, \mu_{X_t}, P_t, \theta^{*}_t\right)=0, \quad d t \otimes d \mu \text {a.s.}
$$    
\end{lemma}

According to the above lemma, the optimal control problem in the form of equation \eqref{5.13} can be solved by solving the random Pontryagin maximum principle. In the sense of mean field, the optimal control corresponding to equation \eqref{5.13}
\begin{equation} \label{5.22}
\boldsymbol{Q}\left(\boldsymbol{\theta}^*\right)_t:=\mathbb{E}_{\mu_0} \nabla_\theta H\left(x_t^{\boldsymbol{\theta}^*}, p_t^{\boldsymbol{\theta}^*}, \theta_t^*\right)=0.
\end{equation}
Here, the Hamiltonian function $H$is shown in definition \ref{Definition 5.9}, $p_t$is called conjugate variable, $L(t, x, \dot{x}, \theta)$is the integrand of the Onager -Machlup action function \eqref{MVSDEOM} corresponding to the McKean-Vlasov stochastic differential equation (SDE).

We say that the equation \eqref{5.22} is the equation satisfied by the optimal control $\theta^{*}$of the mean field limit McKean-Vlasov SDE \eqref{5.11}. Next, we consider the possible correspondence between the most probable transition pathways of the stochastic interacting particle system \eqref{paticle system} and the McKean-Vlasov stochastic dynamical system \eqref{5.11} in the functional sense of Onager-Machlup action.

\begin{definition} \label{Definition 5.11}(\textbf{Stable Mapping}\cite{han2019mean}) \quad For the mapping $F: U \rightarrow V$, constant $\rho>0$  and $x \in U $, make $S_\rho(x):=\left\{y \in U:\|x-y\|_U \leq \rho\right\} $. If there is a constant $K_\rho>0$  for all $y, z \in S_\rho(x)$, there is
$$
\|y-z\|_U \leq K_\rho\|F(y)-F(z)\|_V.
$$
We say that the mapping $F$ is stable on $S_\rho(x)$.
\end{definition}

\begin{remark}
If $F$ is stable on $S_\rho(x)$, then obviously it has at most one solution on $S_\rho(x)$ for $F=0$. If $D F\left(x^*\right)$ exists, then it is non-singular.
The following statement establishes a stronger version of this: if $D F(x)$ exists for any $x \in S_\rho\left(x^*\right)$, then it must be non-singular.
\end{remark}

Supppse  $(\Theta, \mathcal{F}, \mathbb{P})$ is a probability space, $\left\{F_N(\theta): N \geq 1, \theta \in \Theta\right\}$ is the family of mappings from $\Theta$to $\mathbb{R}$such that for each $x$, $\theta \mapsto F_N (\theta) (x)$ is $\mathcal{F} $- 
 measurable.

\begin{assumption} \label{Hypothesis 5.3}Now let's make the following assumptions about mapping $Q$:

(A1)(stability) \quad exists $\theta^* \in \Theta$ such that $Q\left(\theta^*\right)=0$, and for some $\rho>0$, $Q$ is stable on $S_\rho\left(\theta^*\right)$.

(A2)(uniform convergence in probabilities) \quad for all $N \geq 1$, for all $\theta \in S_\rho\left(\theta^*\right)$, $D Q(\theta)$ and $D Q_N(\theta)$ exist almost everywhere, and
$$
\begin{aligned}
& \mathbb{P}\left[\left\|Q(\theta)-Q_N(\theta)\right\|_\mathbb{R}\geq s\right] \leq r_1(N, s), \\
& \mathbb{P}\left[\left\|D Q(\theta)-D Q_N(\theta)\right\|_\mathbb{R} \geq s\right] \leq r_2(N, s),
\end{aligned}
$$
The above formula holds for some real-valued functions $r_1, r_2$, and satisfies that when $N \rightarrow \infty$, $r_1(N, s), r_2(N, s) \rightarrow 0$.

(A3)(uniform Lipschitz derivative)\quad exists $K_L>0$, such that for all $\theta_1, \theta_2 \in S_\rho\left(\theta^*\right)$,
$$
\left\|D Q_N(\theta_1)-D Q_N(\theta_2)\right\|_\mathbb{R}\leq K_L\|\theta_1-\theta_2\|_\Theta,  \quad \mathbb{P} \text {- a.s. }
$$
\end{assumption}

\begin{theorem} \label{Theorem 5.5}(\textbf{Correspondence relation of the most probable transition pathways for a stochastic interacting particles in the sense of mean field}) \quad If the assumption \ref{Hypothesis 5.3} (A1)-(A3) holds. So, there is a constant $s_0>0$, for each $s \in\left(0, s_0\right]$ and $N \geq 1$, there exists a measurable set $\theta_N(s) \subset \Theta$, there exist two real-valued functions $r_1, r_2$, And satisfy that when $N \rightarrow \infty$, $r_1(N, s), r_2(N, s) \rightarrow 0$. Make $\ \ mathbb {P} \ left [\ theta_N (s) \ right] \ geq $$1 - r_1 (N, s) - r_2 (N, s) $, and for each $\ theta \ \ theta_N in $(s), with the following error:
$$
\left\|Q_N\left(\theta\right)-Q(\theta^*)\right\|<s .
$$
In addition, $D Q_N(\theta)$is non-singular and satisfies the following relation:
$$
\left\|D Q_N(\theta)^{-1}\right\|_\Theta\leq 2\left\|D Q\left(\theta^*\right)^{-1}\right\|_\Theta .
$$
In particular, we say $D Q_N (\theta) $ is stable on $S_{\rho_0}\left(\theta^*\right)$, including constant $\rho_0 $ meet
$$\rho_0 \leq \min \left(\rho, \frac{1}{4}\left(K_L\left\|D Q\left(\theta^*\right)^{-1}\right\|_\Theta\right)^{-1}\right).$$
Here $K_{\rho_0}$is the stability constant that satisfies the definition \ref{Definition 5.11}, the specific value is $K_{\rho_0}=4\left\|D Q\left(\theta^*\right)^{-1}\right\|_\Theta$.
\end{theorem}

\begin{remark}
Although the result of the theorem does not directly achieve the approximation of the maximum possible migration orbit of the particle system, it shows that the solutions of the core equations in the Pontryagin maximum principle have correspondence. Therefore, it is reasonable to expect that the correspondence between the most possible transfer orbits can be solved when the solution equations corresponding to the optimal control $\theta^*$ and $\theta^{N}$are approximated. This will be the focus of future research.
\end{remark}
\begin{proof} For $s>0$, define
$$
\begin{aligned}
\theta_N(s):= & \left\{\theta \in \Theta: \left\|Q\left(\theta^*\right)-Q_N(\theta)\right\|_\mathbb{R}<s\right. \\
& \left.\text {and } \left\|D Q\left(\theta^*\right)-D Q_N(\theta)\right\|_\Theta<s\right\} .
\end{aligned}
$$
It is observed that $D Q_N(\theta)$ is measurable, and according to the assumption \ref{Hypothesis 5.3}(A3), we have $\mathbb{P}\left[\theta_N(s)\right] \geq 1-r_1(N, s)-r_2(N, s)$. Now, choose $s$ that is small enough to make $s \leq$ $s_0=\frac{1}{2}\left\|D Q\left(\theta^*\right)^{-1}\right\|_\Theta^{-1}$. For each $\theta \in \theta_N(s)$, it is known from the Banach lemma \cite{tromba1972morse} that $D Q_N(\theta)$ is non-singular, and 
$$
\left\|D Q_N(\theta)^{-1}\right\|_\Theta \leq \frac{\left\|D Q\left(\theta^*\right)^{-1}\right\|_\Theta}{1-\frac{1}{2}}=2\left\|D Q\left(\theta^*\right)^{-1}\right\|_\Theta .
$$
Finally, from the literature \cite{han2019mean}[Proposition 5], we can derive the stability of $D Q_N(\theta)$ on $S_{\rho_0}\left(\theta^*\right)$, where $\rho_0 \leq \min \left(\rho, \frac{1}{4}\left(K_L\left\|D Q\left(\theta^*\right)^{-1}\right\|_\Theta\right)^{-1}\right)$, stability constant is $K_{\rho_0}=4\left\|D Q\left(\theta^*\right)^{-1}\right\|_\Theta$. 
\end{proof}

The above theorem tells us that when the solution $\boldsymbol{\theta}^*$ to the optimal control problem of the mean field limit McKean-Vlasov stochastic differential equation is stable, for a sufficiently large number of particles $\rm N$, We are very likely to find a random variable $\boldsymbol{\theta}^N$ in its neighborhood, which is a stationary solution to the optimal control problem of the mean field limit equation. With the solution of the mean field optimal control problem, we can use the Pontryagin maximum principle to solve the Forward-Backward Stochastic Differential Equation (FBSDE), so as to find the most probable transition pathways of the system. The above theorem establishes the estimation of the control variable of the most probable transition pathways of the particle system. Although the approximation of the orbit is not realized directly, the result of the theorem illustrates the correspondence between the stochastic interacting particle system and the most probable transition pathways  of the mean field limit McKean-Vlasov stochastic differential equation.

\section{A Machine Learning Framework for Solving High-dim Interacting Particle System}%%% 写算法的部分

\subsection{Problem Setting: Interacting Particle Systems and the Mean Field Limit}

We consider the problem of approximating a stochastic interacting particle system and its mean field limit described by a McKean–Vlasov stochastic differential equation (SDE), under Brownian noise. Let $B = \left(B_t\right)_{0 \leq t \leq T}$ be a standard $k$-dimensional Brownian motion defined on a filtered probability space $(\Omega, \mathcal{F}, \mathbb{P})$ with the natural filtration $\mathbb{F} = \left(\mathcal{F}_t\right)_{0 \leq t \leq T}$. For any random variable/vector or stochastic process $X$, we denote its distribution by $\mu_X$.

We use $\mathscr{P}_2\left(\mathbb{R}^{d}\right)$ to denote the $2$-Wasserstein space of Borel probability measures on $\mathbb{R}^{d}$ with finite second moments. Equipped with the Wasserstein-2 metric $W_2$, this space forms a complete and separable metric space. Throughout this work, we assume the marginal distribution $\mu_t \in \mathscr{P}_2(\mathbb{R}^d)$ for all $t$.

Following the formulation by Sznitman \cite{burkholder1991topics}, we focus on the McKean–Vlasov SDE of the form:
\begin{equation} \label{eq:mkv-sde}
\begin{aligned}
d X_t &= \left[ \int_{\mathbb{R}^d} b(X_t, y) \, \mu_t(dy) \right] dt + \sigma \, dB_t, \quad 0 \leq t \leq T, \\
X_0 &\sim \mu_0,
\end{aligned}
\end{equation}
where
\begin{description}[leftmargin=2em,labelindent=0em]
    \item[$X_t \in \mathbb{R}^d$] the state of an individual agent at time \(t\);
    \item[$\mu_t := \mu_{X_t} \in \mathcal{P}_2(\mathbb{R}^d)$] the distribution of the agents at time \(t\);
    \item[$b : \mathbb{R}^d \times \mathbb{R}^d \to \mathbb{R}^d$] a Lipschitz continuous drift function modeling pairwise interactions;
    \item[$\sigma > 0$] the diffusion intensity, modeling the level of stochasticity.
\end{description}

This formulation captures the evolution of a large population of interacting agents under both mutual influence and stochastic perturbations. The system evolves under the averaged interaction field generated by the current population distribution, and thus exhibits mean field behavior. In the limit of an infinite number of agents, the empirical distribution converges (under appropriate assumptions) to the solution $\mu_t$ of the above nonlinear SDE. Our objective is to learn the associated control and density functions of this stochastic mean field system using neural approximation techniques in high dimensions.

%%接下来我们致力于求解对应的粒子系统（随机高维的MFGs）
\subsection{Mean-Field Optimal Control Formulation via Onsager--Machlup Functional}

Next, we consider the problem of critical transition governed by the stochastic McKean--Vlasov dynamics. Following the Onsager--Machlup action functional and incorporating the corresponding Lagrangian $L(t, \phi, \dot{\phi})$, we pose the following mean-field optimal control problem:
\begin{equation}\label{eq:mf-control}
\underset{\theta \in \Theta }{\operatorname{min}} \quad J(\theta)=\mathbb{E}_{\mu_0}\left\{\int_0^T L\left(t, \phi_t, \dot{\phi}_t; \theta \right) dt + g\left(\phi_T, \mu_{\phi_T}\right)\right\}.
\end{equation}
Here
\begin{description}[leftmargin=2em,labelindent=0em]
    \item $\theta \in \Theta$ is a set of admissible control parameters, e.g., parameterizing the policy or trajectory $\phi$;
    \item $L: [0, T] \times \mathbb{R}^d \times \mathscr{P}_2(\mathbb{R}^d) \to \mathbb{R}$ is the running cost functional derived from the Onsager–Machlup framework;
    \item $g: \mathbb{R}^d \times \mathscr{P}_2(\mathbb{R}^d) \to \mathbb{R}$ is the terminal cost function that enforces desired distributional behavior at the final time;
    \item $\mu_0$ denotes the initial distribution of the agent state.
\end{description}

This formulation aims to find a control strategy, encoded by $\theta$, that minimizes the expected total cost, which consists of the trajectory-wise Onsager–Machlup action and a terminal regularization cost. The optimization is performed over a function space of admissible trajectories or control policies that influence the distributional evolution of the system.

Importantly, this stochastic optimal control problem over probability measures naturally extends the classical deterministic mean field game setting to a fully stochastic, measure-dependent control framework.

\subsection{GAN-inspired Optimization for Mean Field Control Problem}

While generative adversarial networks (GANs) are originally designed for learning data distributions, recent advances have shown a deep connection between generative modeling and mean field control. In particular, the generator in a GAN can be viewed as a parameterized stochastic simulator—analogous to a policy in MFC—while the discriminator plays the role of a value function that guides the optimization.

%%%% 解释这个式子和 （5.2） 的关系 ，把H带进去

To this end, we propose to adopt a GAN-style saddle-point training procedure to solve the MFC problem. Specifically, we parameterize the control strategy using a neural generator $\phi^\theta$, and optimize it against a value-function-based critic $\varphi^\psi$, leading to the following \textbf{min-max formulation}:
\begin{equation}
\min_\theta \max_\psi \quad \mathbb{E}\left[ \int_0^T \left( \nabla_x \varphi^\psi(t, X_t^\theta) \cdot \dot{X}_t^\theta - L(t, X_t^\theta, \dot{X}_t^\theta, \mu_t^\theta) \right) dt - g(X_T^\theta, \mu_T^\theta) \right].
\end{equation}
This formulation draws direct inspiration from Wasserstein GANs and provides a scalable approach for solving high-dimensional mean field control problems without discretizing the underlying state space. The following table summarizes the conceptual and mathematical correspondence between these two frameworks:
\begin{table}[htbp]
\centering
\small
\caption{Components of Mean Field Control (MFC) and its conceptual analogy to Generative Adversarial Networks (GANs).}
\renewcommand{\arraystretch}{1.1}
\begin{adjustbox}{max width=\linewidth}
\renewcommand{\arraystretch}{1.25}
\begin{tabular}{lll}
\toprule
\textbf{Component} & \textbf{\textit{Mean Field Control (MFC)}} & \textbf{Generative Adversarial Networks (GANs)} \\
\midrule
Generator & \textbf{Control policy}  & Neural generator \( G(z) \) \\
Critic & \textbf{Value or cost functional}  & Discriminator \( D(x) \) \\
Objective & \textbf{Expected cost minimization}  & Distributional discrepancy minimization \\
Data & \textbf{Trajectory distribution} & Data distribution \( \mathbb{P}_{\text{data}} \) \\
Output & \textbf{Controlled trajectories} & Generated samples \( G(z) \) \\
Training & \textbf{Alternating optimization of policy and value} & Alternating optimization of generator and discriminator \\
Goal & \textbf{Population-level optimal control in stochastic settings} & Learning complex data distributions \\
\bottomrule
\end{tabular}
\label{tab:MFC_GAN_comparison}
\end{adjustbox}
\end{table}

This correspondence reveals that solving a Mean Field Control (MFC) problem is structurally analogous to training a GAN, where the generator plays the role of a parameterized control policy steering the evolution of a population of agents, and the discriminator corresponds to the value function assessing the quality of the induced state distribution. The optimization over control policies mirrors the generator’s learning to match a target behavior, while the value function acts as a critic evaluating the cost-to-go and guiding policy improvement. This analogy enables the application of adversarial learning principles to solve high-dimensional MFC problems and suggests new directions for designing control-aware generative models based on population dynamics.
\subsection{A GAN-style Min-Max Formulation for Mean-Field Optimal Control}

Inspired by the above structural similarities between generative adversarial networks (GANs) and mean-field-type control (MFC) problems, we present a GAN-style saddle-point reformulation of the stochastic optimal control problem over probability measures. Recall the mean-field optimal control problem:
\begin{equation}\label{eq:mf-control}
\underset{\theta \in \Theta }{\operatorname{min}} \quad J(\theta)=\mathbb{E}_{x_0 \sim \mu_0}\left\{\int_0^T L\left(t, \phi_t^\theta(x_0), \dot{\phi}_t^\theta(x_0)\right) dt + g\left(\phi_T^\theta(x_0), \mu_{\phi_T^\theta}\right)\right\},
\end{equation}
where the control parameter $\theta$ defines a family of admissible trajectories $\phi^\theta$ starting from the initial law $\mu_0$. The function $L$ represents a running cost derived from the Onsager–Machlup action functional, and $g$ is a terminal cost that encourages certain distributional behaviors at the final time $T$. To reformulate this as a GAN-style problem, we introduce a discriminator function $\mathcal{D}_\psi: [0, T] \times \mathbb{R}^d \to \mathbb{R}$, parameterized by $\psi \in \Psi$, which estimates the value function or cost density. The resulting saddle-point formulation reads:
\begin{equation}\label{eq:gan-mfc}
\min_{\theta \in \Theta} \max_{\psi \in \Psi} \quad \mathbb{E}_{x_0 \sim \mu_0} \left[ \int_0^T \left( L(t, \phi_t^\theta(x_0), \dot{\phi}_t^\theta(x_0)) + \mathcal{D}_\psi(t, \phi_t^\theta(x_0)) \right) dt + g(\phi_T^\theta(x_0), \mu_{\phi_T^\theta}) \right].
\end{equation}

In this formulation, the \textbf{generator} $\theta$ seeks to minimize the expected control cost by generating optimal trajectories. The \textbf{discriminator} $\psi$ acts as a value function or cost critic, adversarially maximizing the objective to provide meaningful gradients. The interaction between $\theta$ and $\psi$ resembles that of a GAN: the generator adapts its policy to fool the discriminator into assigning low cost to its induced distribution. If the terminal cost $g$ is chosen to be a Wasserstein-type divergence from a desired terminal law $\mu_T$, then the saddle-point problem reduces to a regularized Wasserstein GAN:
\begin{equation}
\min_{\theta \in \Theta} \max_{\psi \in \text{Lip}_1} \quad \mathbb{E}_{x \sim \mu_{\phi_T^\theta}}[\psi(x)] - \mathbb{E}_{y \sim \mu_T}[\psi(y)] + \lambda \, \mathbb{E}_{x_0 \sim \mu_0} \left[\int_0^T L(t, \phi_t^\theta(x_0), \dot{\phi}_t^\theta(x_0)) dt \right].
\end{equation}

This correspondence reveals that solving a mean-field optimal control problem can be cast as a generative adversarial optimization, wherein the generator learns optimal stochastic dynamics and the discriminator guides the learning via an implicit cost structure. This connection offers a powerful approach to solving high-dimensional MFC problems using tools from deep learning and adversarial training. In Figure \ref{fig:om-gan-framework}, we show a schematic diagram of the framework process to solve the MFC problem based on the Wasserstein GAN.

\section{Solving Mean-Field Control via Onsager–Machlup Wasserstein GAN (OM-GAN-Control)}

To address the high-dimensional mean-field optimal control (MFC) problem under stochastic dynamics, we propose a GAN-inspired algorithm that integrates the Onsager–Machlup action functional into the training of a generator–discriminator pair. This approach allows us to leverage adversarial learning to approximate optimal migration trajectories in distribution space.

\subsection{Algorithmic Framework}

In this section, we proposed method, \textbf{OM-GAN-Control}, formulates the MFC problem as a min-max optimization problem inspired by Wasserstein GANs. The generator parameterizes the stochastic control policy, while the discriminator acts as a critic approximating the Wasserstein distance between the evolving state distribution and the desired terminal distribution.

\begin{figure}[H]
    \centering
    \includegraphics[width=0.5\linewidth]{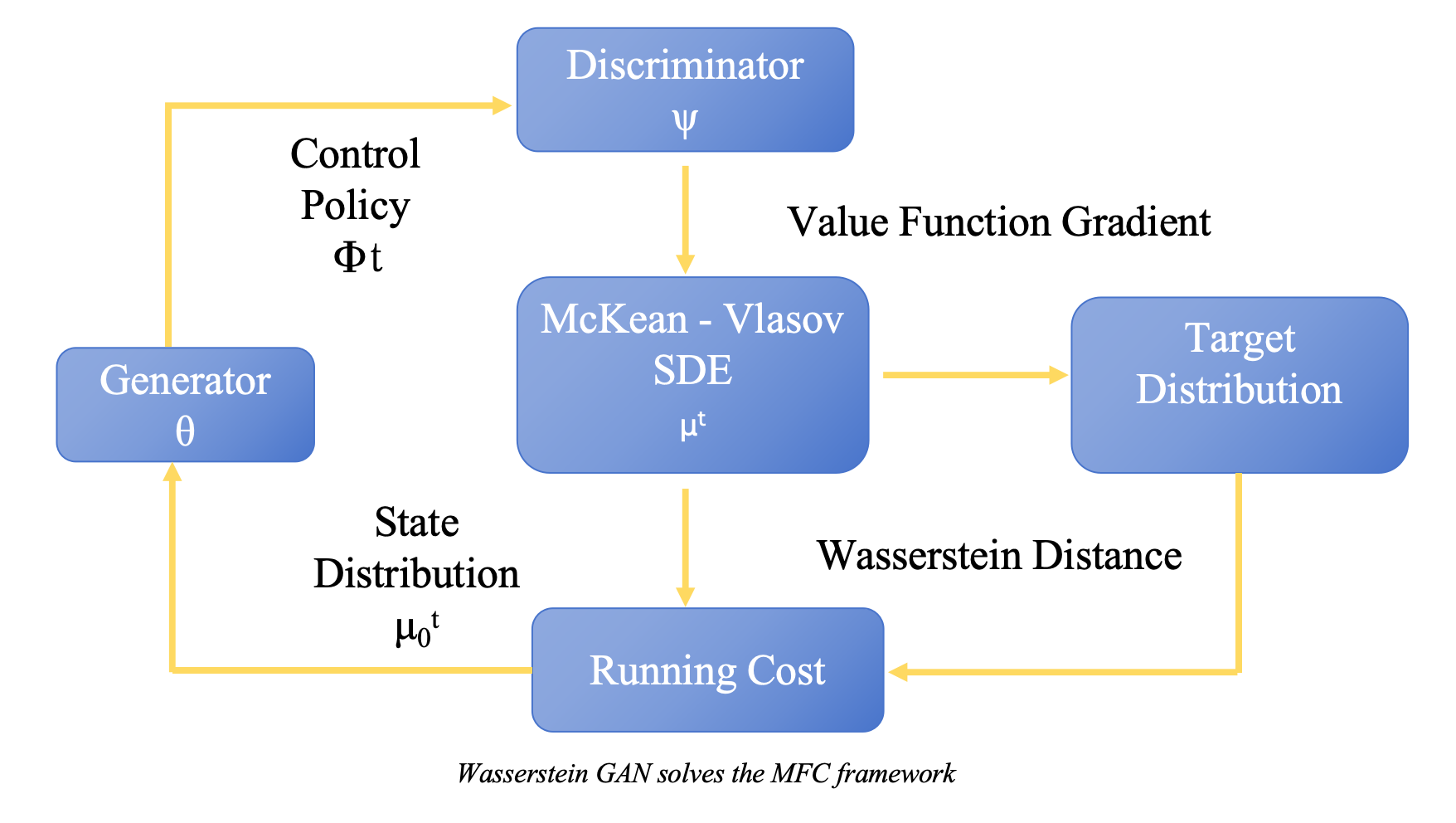}
    \caption{Framework of OM-GAN-Control: Using a Wasserstein GAN formulation to solve mean-field optimal control under Onsager--Machlup theory.}
    \label{fig:om-gan-framework}
\end{figure}

\subsection{Related Works}
\paragraph{Reinforcement Learning under SDE Dynamics and Mean-Field Control.}
Recent advances in reinforcement learning (RL) for systems governed by stochastic differential equations (SDEs) have established a promising connection with mean-field control (MFC) theory. Unlike traditional discrete-time RL, these methods operate in continuous-time environments where the state evolution follows Itô dynamics. In particular, \cite{li2021deep} proposes a deep FBSDE-based policy gradient method for solving continuous-time stochastic control problems, which can be extended to mean-field settings. \cite{jia2022policy} develops a model-free policy optimization method under unknown SDE dynamics, leveraging stochastic approximation and Girsanov transformations. For MFC problems, these approaches are particularly relevant when the drift depends on the law of the state, and the objective involves population-level regularization or constraints. \cite{wang2022deep} incorporates mean-field interactions into the RL framework, using neural SDE solvers and empirical distribution approximations. These methods do not require knowledge of the explicit form of the dynamics or reward functions, making them suitable for high-dimensional systems. Our proposed approach builds on this direction by introducing a generative adversarial mechanism to implicitly learn both the control policy and the associated path distribution, while minimizing the Onsager--Machlup functional that characterizes the most likely paths under SDE noise.

\paragraph{GAN-based Approach for MFC Problems.}
Our work builds upon the emerging idea of solving mean-field control (MFC) problems using generative adversarial learning, inspired by recent advances in Wasserstein GANs (WGANs) \cite{arjovsky2017wasserstein, gulrajani2017improved}. Unlike prior works on GAN-based MFG solvers that treat the value function or density as the generator \cite{lin2021apac}, we directly model the control policy through a neural generator, aligning with the control-theoretic objective of optimizing a path-dependent functional. The discriminator, on the other hand, approximates the value gradient or dual potential that guides optimal trajectories, thereby serving as a feedback signal for adversarial training.

A key distinction in our framework lies in the objective: instead of minimizing a residual or fitting a PDE, we minimize the Onsager--Machlup action functional \cite{onsager1953fluctuations, chen2021likelihood} derived from the stochastic path space, which characterizes the most probable evolution under the McKean--Vlasov dynamics. This variational perspective introduces a principled cost structure in the generator's optimization, bridging the gap between optimal control theory and adversarial generative modeling \cite{fleming2006controlled}. Moreover, by using empirical trajectory distributions, we avoid explicit discretization of the state space and can scale to high-dimensional control problems, thus complementing the limitations of traditional shooting methods and FBSDE-based solvers \cite{han2018solving, germain2022neural}.

\section{Numerical Experiments}
We evaluate the performance of \textbf{OM-GAN-Control} on a class of mean-field optimal control problems governed by nonlocal dynamics and two-point boundary conditions. These problems arise in the context of rare event modeling and metastable transition pathways, where the Onsager--Machlup-type action functional is used to characterize the most probable path between given initial and terminal states. %We compare the behavior of the optimal control and induced trajectories across various choices of the diffusion parameter \( \sigma \), and validate the method against analytically solvable low-dimensional examples. Additional high-dimensional results are presented in Appendix~B.
\subsection{Experimental Setup}

\paragraph{Problem Formulation.}
The MFC problem we consider is given by
\begin{equation}\label{eq:om_control}
    \begin{cases}
    \displaystyle \underset{\theta \in \Theta }{\operatorname{min}} \quad \mathbb{E}_{\mu_0} \left[\frac{1}{2}\int_{0}^{T} \left(\theta^{2}+\operatorname{div}_{X}f(t, X_t, \mu_{X_t})\right)dt + g(X(T), \mu_{X(T)})\right],\\[1.2ex]
    \text{subject to} \quad \dot{X}_t = h(X_t) \left[ E(X_t) \right] + \sigma \theta, \quad X(0) = x_0, \quad X(T) = x_T,
    \end{cases}
\end{equation}
where \( \mu_{X_t} \in \mathscr{P}_2(\mathbb{R}^d) \) denotes the population density at time \( t \), and the control variable \( \theta \) is applied multiplicatively through the diffusion coefficient. The dynamics are nonlocal, driven by a functional dependence on both the distribution \( \mu_{X_t} \) and the expectation \( E(X_t) \), representing mean-field interactions. The inclusion of the divergence term \( \operatorname{div}_X f \) accounts for non-equilibrium transport effects.

\paragraph{Path-space Control Problem (Onsager--Machlup Formulation).}
The optimal control problem can be equivalently formulated as minimizing the Onsager--Machlup action functional:
\begin{equation} \label{eq:om_control_final}
\begin{cases}
\displaystyle \underset{\theta \in \Theta}{\operatorname{min}} \quad \mathbb{E}_{\mu_0} \left[ \int_0^T L\left(t, X_t, \dot{X}_t \right) \, dt + g(X_T, \mu_{X_T}) \right], \\[1.2ex]
\text{subject to} \quad \dot{X}_t = f(t, X_t, \mu_{X_t}) + \sigma \theta_t, \quad X(0) = x_0, \quad X(T) = x_T,
\end{cases}
\end{equation}
where the Lagrangian is defined as
\begin{equation}
L\left(t, \phi, \dot{\phi}\right) = \frac{1}{2} \left[ \mathrm{B}^{-1} 
\left| \dot{\phi}_t - f\left(t, \phi_t, \mu_{\phi_t} \right) \right|^2 \right] + \frac{1}{2} \operatorname{div}_x f\left(t, \phi_t, \mu_{\phi_t} \right),
\end{equation}
and the divergence term is given by
\begin{equation}
\operatorname{div}_x f = \sum_{i=1}^{d} \partial_{x_i} f_i \left(t, \phi_t, \mu_{\phi_t} \right),
\end{equation}
with \( \mathrm{B} = \sigma \sigma^{*} \) representing the diffusion tensor. The functional in \eqref{eq:om_control_final} corresponds to the Onsager--Machlup action over controlled stochastic trajectories and penalizes both deviation from the drift \( f \) and density-induced compressibility effects. 

\paragraph{Explicit Form of OM-GAN-Control Problem.}
We now substitute the Lagrangian and dynamics into the control problem to obtain the following explicit form:
\begin{equation} \label{eq:om_control_explicit}
\begin{cases}
\displaystyle \underset{\theta \in \Theta}{\operatorname{min}} \quad \mathbb{E}_{\mu_0} \left[ \int_0^T \left( 
\frac{1}{2} \theta_t^{\top} \mathrm{B}^{-1} \theta_t + \frac{1}{2} \sum_{i=1}^{d} \partial_{x_i} f_i(t, X_t, \mu_{X_t})
\right) dt + g(X_T, \mu_{X_T}) \right], \\[1.5ex]
\text{subject to} \quad \dot{X}_t = f(t, X_t, \mu_{X_t}) + \sigma \theta_t, \quad X(0) = x_0, \quad X(T) = x_T,
\end{cases}
\end{equation}
where the control variable \( \theta_t \in \mathbb{R}^{d} \) is regularized by the inverse diffusion-weighted quadratic cost 
\( \theta_t^{\top} \mathrm{B}^{-1} \theta_t \), and the divergence term penalizes the local compressibility of the drift field \( f \). The terminal cost \( g(X_T, \mu_{X_T}) \) measures the deviation of the population from the target distribution.

\begin{remark}
\textbf{Isotropic Diffusion:} If \( \mathrm{B} = \sigma^2 I \) represents isotropic diffusion (where \( I \) is the identity matrix), the term \( \theta_t^{\top} \mathrm{B}^{-1} \theta_t \) simplifies to \( \frac{1}{\sigma^2} \|\theta_t\|^2 \). This corresponds to a standard quadratic regularization term often used in optimal control and reinforcement learning, where \( \theta_t \) is the control input and \( \|\theta_t\|^2 \) penalizes large control actions to avoid over-exploration.
\end{remark}
\begin{remark}
 \textbf{Particle Interaction Systems:} Because the drift term \( f(t,x,\mu) \) represents a particle interaction system, we can model it as:
    \[
    f(t,x,\mu) = -\nabla_x U(x) + \int K(x,y) \mu(dy),
    \]
    where \( U(x) \) is a potential energy function governing the interaction between particles, and \( K(x,y) \) is a kernel function that represents pairwise interactions between particles. In this case, the control problem accounts for both the interaction of particles and the mean-field influence of the population density \( \mu \), leading to a more complex drift field that can be used to model, for example, crowd dynamics or collective behavior in a system of agents.  
\end{remark}

   \textbf{Optimization via GAN-based Methods:} This form is particularly convenient for optimization with GAN-based methods. The expected value in the objective function can be approximated by sampling paths \( \{X_t\}_{t \in [0,T]} \) from the controlled dynamics. The Wasserstein GAN framework allows us to use the discriminator to approximate the Wasserstein distance between the distribution of trajectories generated by the model and the target distribution. The generator network then seeks to minimize this distance, learning an optimal control strategy \( \theta_t \) through backpropagation.
    
    The structure of the functional is naturally suited for the Wasserstein GAN approach because it involves both local smoothness (via the divergence term) and the control regularization, which aligns with the goals of a GAN — to learn a distribution that closely matches a reference distribution. Furthermore, the clear separation between the control term and the system dynamics makes it easy to integrate the control mechanism into a GAN framework.

\textbf{GAN Loss Function for OM-GAN-Control.}
In the GAN framework, the generator network \( G_\theta \) generates the control trajectory \( \theta_t \), and the discriminator \( D_\omega \) evaluates how well the generated trajectory matches the desired behavior according to the Onsager–Machlup functional. The loss function for the generator and discriminator in the context of the \textbf{OM-GAN-Control} problem can be written as:

\begin{equation}
    \mathcal{L}_{\text{GAN}} = \mathbb{E}_{\mu_0} \left[ \int_0^T L\left(t, X_t, \dot{X}_t \right) \, dt + g(X_T, \mu_{X_T}) \right],
\end{equation}
where \( L(t, X_t, \dot{X}_t) \) is the Lagrangian described earlier, and the discriminator \( D_\omega \) is used to monitor the Wasserstein distance between the distribution of generated trajectories and the target distribution. The generator \( G_\theta \) aims to minimize this loss, which leads to an optimal control trajectory that governs the system's dynamics over time. The discriminator \( D_\omega \) plays the role of ensuring that the generated paths adhere to the desired distribution in terms of both the control regularization and the physical dynamics of the system.

\textbf{Neural Parameterization.} The control function \( \theta_t \) is approximated using a generator network \( G_\theta \), while a discriminator network \( D_\omega \) is used to estimate the Wasserstein distance between the reference path measure (uncontrolled) and the optimally controlled path measure. Both networks are implemented as residual neural networks (ResNets), which consist of three hidden layers, each with 512 units. Skip connections are used with a weight of 0.5 to stabilize the training process. The generator \( G_\theta \) uses the ReLU activation function, while the discriminator \( D_\omega \) uses the Tanh activation function to ensure smooth gradients and improved learning efficiency.

% \paragraph{Evaluation.}

% For high-dimensional experiments, the interaction terms \( f(t, X_t, \mu_{X_t}) \) are specifically designed to affect only the first two coordinates. This allows both visual inspection and quantitative assessment of the optimization process while still capturing essential features of the high-dimensional problem. 

\subsection{OM-GAN-Control Algorithm }
We aim to solve the Mean Field Control (MFC) problem, where the control function $\theta_t$ governs the system dynamics, and the optimal control is obtained by minimizing the Onsager-Machlup (OM) functional.

\begin{algorithm}[htbp]
\caption{OM-GAN-Control: Solving MFC via Wasserstein GAN based on Onsager--Machlup Functional}
\label{alg:OM-GAN-Control}
\begin{algorithmic}[1]
    \Require Initial distribution $\mu_0$, target distribution $\mu_T$, time horizon $T$, step size $\Delta t$, iterations $M=T/\Delta t$
    \Require Initialize generator parameters $\theta$, trajectory parameters $\phi$
    \Require Batch size $N$, learning rate $\eta$
    \Ensure Optimal control parameters $\theta^*$, Most Probable Transition Pathway (MPTP) $x^*$
    \While{not converged}
        \State Sample batch of initial data $\{x_0^{(i)}\}_{i=1}^N \sim \mu_0$
        \State Generate controlled trajectories $\{\phi_k^{(i)}\}$ via Euler scheme:
        \begin{equation*}
            \phi_{k+1}^{(i)} = \phi_k^{(i)} + f(t_k, \phi_k^{(i)}, \mu_k^{(N)}) \Delta t 
            + \sigma^{(i)} \sqrt{\Delta t}\,\theta_k^{(i)}.
        \end{equation*}
        \State Compute empirical distribution $\mu_t^{(N)}$ from $\{\phi_t^{(i)}\}$
        \State Evaluate Onsager--Machlup Lagrangian:
        \begin{align*}
            L_t^{OM} &= \tfrac{1}{2}\Big\| B^{-1}\big(\dot{\phi}_t - f(t,\phi_t,\mu_{\phi_t})\big)\Big\|^2 
            + \tfrac{1}{2}\,\mathrm{div}_x f(t,\phi_t,\mu_{\phi _t}).
        \end{align*}
        \State Compute total loss:
        \begin{align*}
            \mathcal{L}_G &= \sum_{k=0}^M \Bigg( 
            \tfrac{1}{2}\sum_{i=1}^N \theta_k^{(i)\top} B^{-1} \theta_k^{(i)}
            + \tfrac{1}{2}\sum_{i=1}^N \partial_{x_i} f_i(t_k,\phi_k,\mu_k^{(N)})
            + g(\mu_T, \mu_{\phi_T}) \Bigg).
        \end{align*}
        \State Update control parameters, where $\theta$ is defined as $\left| \dot{\phi}_t - f\left(t, \phi_t, \mu_{\phi_t} \right) \right|$:
        \[
            \theta \gets \theta - \eta_\theta \nabla_\theta \mathcal{L}_G
        \]
        \State Estimate distribution $\hat{\mu}_t^{(N)}$ with updated $\theta$
        \State Compute Wasserstein-2 loss between distributions $\hat{\mu}_t^{(N)}$ and $\mu_t^{(N)}$  :
        \begin{equation*}
            \mathcal{L}_{W_2} = W_2^2 \!\left( \hat{\mu}_t^{(N)}, \mu_t^{(N)} \right)
            = \inf_{\gamma \in \Pi(\hat{\mu}_t^{(N)}, \mu_t^{(N)})} 
            \int_{\mathcal{X}\times\mathcal{X}} \|x-y\|^2 \, d\gamma(x,y).
        \end{equation*}
        \State Update trajectory parameters:
        \[
            \phi \gets \phi - \eta_\phi \nabla_\phi \mathcal{L}_{W_2}
        \]
    \EndWhile
    \State \textbf{Output:} Optimal control parameters $\theta^*$ , the Most Probable Transition Pathway(MPTP) $x^{*}$.
\end{algorithmic}
\end{algorithm}

\subsection{Application to FitzHugh-Nagumo Neuronal Dynamics Model}

\subsubsection{Model Formulation and Stochastic Perturbation}

In order to investigate the noise-induced transitions between metastable states in a Neuronal Dynamical system, we consider a coupled slow-fast system with $N$ nodes. Each node consists of a fast variable $x_i$ and a slow variable $y_i$, governed by the following deterministic system:
\begin{equation}
\begin{aligned}
\dot{x}_i &= F(x_i, y_i, \mu) + \frac{K}{N} \sum_{j=1}^{N} w_{ij} H(x_i, y_i, x_j, y_j), \\
\dot{y}_i &= G(x_i, y_i, \mu),
\end{aligned}
\quad i=1, \ldots, N,
\label{eq:deterministic}
\end{equation}
where $F$ and $G$ describe the intrinsic node dynamics parameterized by a bifurcation parameter $\mu$. The coupling matrix $W = \{ w_{ij} \}$ encodes the network topology, $K$ controls the coupling strength, and $H$ is the coupling function. This general form can describe various systems, including neuronal dynamics, gene regulatory networks, and ecological models.

To capture stochastic transitions between metastable states, we incorporate additive Brownian noise into the model. The addition of noise transforms the system into a set of stochastic differential equations (SDEs):
\begin{align}
\begin{aligned}
dx_i &= \Bigg[ F(x_i, y_i, \mu)
+ \frac{K}{N} \sum_{j=1}^{N} w_{ij} H(x_i, y_i, x_j, y_j) \Bigg] dt
+ \sigma_x\, dW_i^{x}(t), \\
dy_i &= G(x_i, y_i, \mu)\, dt
+ \sigma_y\, dW_i^{y}(t), 
\qquad i = 1, \ldots, N.
\end{aligned}
\label{eq:coupled_sde_xy}
\end{align}

\noindent
Where $\sigma_x, \sigma_y > 0$ denote the noise intensities acting on the 
$x$- and $y$-components, respectively. The inequality $\sigma_y \ll \sigma_x$ characterizes a slow--fast separation, where the $x$-component evolves under stronger stochastic forcing. 
The stochastic processes $W_i^{x}(t)$ and $W_i^{y}(t)$ are 
independent standard Wiener processes satisfying:
\begin{align*}
\mathbb{E}\!\left[ dW_i^{x}(t)\, dW_j^{x}(t) \right] &= \delta_{ij}\, dt, 
&\mathbb{E}\!\left[ dW_i^{y}(t)\, dW_j^{y}(t) \right] &= \delta_{ij}\, dt, 
&\mathbb{E}\!\left[ dW_i^{x}(t)\, dW_j^{y}(t) \right] &= 0,
\end{align*}
indicating that each component and each agent experiences 
\textit{independent and identically distributed} Gaussian perturbations. Under the additive noise setting, we will take $\sigma_x$ and $\sigma_y$  as constants.

\paragraph{Euler--Maruyama discretization.}
For a time step $\Delta t>0$ and $t_k = k\,\Delta t$, the scheme for \eqref{eq:coupled_sde_xy} reads
\begin{equation}
\begin{aligned}
x_{i,k+1} &= x_{i,k} 
+ \Bigg[ F(x_{i,k},y_{i,k},\mu_k)
+ \frac{K}{N}\sum_{j=1}^{N} w_{ij}\, H(x_{i,k},y_{i,k},x_{j,k},y_{j,k}) \Bigg]\Delta t 
+ \sigma_x \sqrt{\Delta t}\,\xi^{x}_{i,k}, \\
y_{i,k+1} &= y_{i,k} 
+ G(x_{i,k},y_{i,k},\mu_k)\,\Delta t 
+ \sigma_y \sqrt{\Delta t}\,\xi^{y}_{i,k},
\qquad i=1,\ldots,N,\; k=0,1,\ldots .
\end{aligned}
\label{eq:euler_xy}
\end{equation}

\noindent
\textit{Independence and identical distribution.}
The innovations $\xi^{x}_{i,k},\xi^{y}_{i,k}\sim\mathcal N(0,1)$ are i.i.d. Gaussian random variables, mutually independent across indices and components, i.e.
\[
\mathbb E[\xi^{x}_{i,k}\xi^{x}_{j,\ell}]=\delta_{ij}\delta_{k\ell},\quad
\mathbb E[\xi^{y}_{i,k}\xi^{y}_{j,\ell}]=\delta_{ij}\delta_{k\ell},\quad
\mathbb E[\xi^{x}_{i,k}\xi^{y}_{j,\ell}]=0 .
\]
Here $\mu_k$ denotes the empirical law at time $t_k$ obtained from $\{(x_{i,k},y_{i,k})\}_{i=1}^N$.

% \begin{equation}
% \begin{aligned}
% d x_i &= \left[ F(x_i, y_i, \mu) + \frac{K}{N} \sum_{j=1}^{N} w_{ij} H(x_i, y_i, x_j, y_j) \right] dt + \sigma_x dW_i^x(t), \\
% d y_i &= G(x_i, y_i, \mu) dt,
% \end{aligned}
% \quad i=1, \ldots, N,
% \label{eq:stochastic}
% \end{equation}
% where $\sigma_x$ denotes the noise intensity acting on the fast variables and $W_i^x(t)$ are independent standard Wiener processes.

% The rationale for adding noise exclusively to the fast variables is three-fold. (1) First, fast variables are most sensitive to instantaneous dynamics; stochastic perturbations in these variables effectively facilitate transitions across potential barriers, allowing noise-induced switching between metastable states. (2) Second, slow variables evolve on much longer timescales, and noise acting on them tends to induce gradual drifts rather than sharp transitions, making them less effective in driving the desired metastable dynamics. (3) Third, from a computational perspective, adding noise solely to the fast variables reduces the overall dimensionality of the stochastic terms and improves numerical stability during simulation.

This modeling choice aligns well with the biophysical interpretation in neuronal systems, where fast variables often represent membrane potentials subject to stochastic synaptic inputs and ion channel fluctuations. Consequently, by focusing on noise in the fast variables, we capture both the essential dynamical features and the computational tractability needed for subsequent numerical investigations.

To ensure that our model is comparable to existing studies and to validate our approach against known benchmarks, we adopt the parameter values from the work of Qin and Lin \cite{QinLin2023}, which investigates tipping points for pulsatile oscillations in dynamical networks. Specifically, we focus on the FitzHugh-Nagumo system, which describes the self-dynamics of each node through the membrane potential $V_i$ (fast variable) and the recovery variable $W_i$ (slow variable). The model equations for node $i$ are given by:
\begin{equation}
\begin{aligned}
\frac{dV_i}{dt} &= V_i (V_i - \theta)(1 - V_i) - W_i + I_{\mathrm{ext}}, \\
\frac{dW_i}{dt} &= \epsilon \left( V_i - \gamma W_i \right),
\end{aligned}
\label{eq:fitzhugh-nagumo}
\end{equation}
where $I_{\mathrm{ext}}$ represents an external current input. For the coupling function, we adopt a sigmoid-type interaction:
\begin{equation}
H(V_i, W_i, V_j, W_j) = \frac{1}{1 + \exp\left[ -\tau \left( V_j - V^* \right) \right]},
\label{eq:sigmoid}
\end{equation}
where $\tau$ controls the steepness of the sigmoid and $V^*$ is a reference potential. We take the same parameter values as in \cite{QinLin2023}, unless otherwise specified, we set:$
\theta = 0.5,\quad \gamma = 1,\quad \epsilon = 0.003,\quad \tau = 1,\quad V^* = 0.$

These parameters ensure that fast-slow dynamics exhibit metastable states and pulsatile oscillations that are characteristic of neuronal networks. The coupling matrix $W = \{ w_{ij} \}$ can be chosen based on network topology; for initial experiments, we use an all-to-all network with uniform weights or alternatively a sparse random network to reflect more realistic neuronal architectures. With this parameterization, the noise-perturbed stochastic model described in Eq.~\eqref{eq:stochastic} is fully specified, providing a solid foundation for numerical simulations of noise-induced metastable transitions and control strategies.

\subsubsection{Numerical Settings}

\begin{table}[htbp]
\centering
\small
\caption{Numerical parameters used in the experiments.}
\renewcommand{\arraystretch}{1.1}
\begin{adjustbox}{max width=\linewidth}
\renewcommand{\arraystretch}{1.25}
\begin{tabular}{ll}
\toprule
\textbf{Parameter} & \textbf{Description / Value} \\
\midrule
$K$ & Coupling strength, chosen from $K \in [0.1,2.0]$; here $K = 1$. \\
$N$ & Number of nodes in the network; experiments start with $N = 10$ and scale up to $N = 100$. \\
$I_{\mathrm{ext}}$ & Constant external input, set to $I_{\mathrm{ext}} = 0.27$. \\
$V_i$ & Initial membrane potential, randomly sampled from $\mathcal{U}[0,1]$. \\
$W_i$ & Initial recovery variable, randomly sampled from $\mathcal{U}[0,0.5]$. \\
\bottomrule
\end{tabular}
\label{tab:numerical_settings}
\end{adjustbox}
\end{table}

According to the above parameter settings, we get the final model as follows:
% \begin{equation}
% \begin{aligned}
% d x_i &= \left[  x_i (x_i - 0.5)(1 - x_i) - y_i + 0.27 + \frac{1}{N} \sum_{j=1}^{N} w_{ij} \frac{1}{1 + \exp \left[\left( - x_j \right) \right]} \right] dt + dW_i^x(t), \\
% d y_i &= 0.003 \left( x_i - y_i \right) dt,
% \end{aligned}
% \quad i=1, \ldots, N,
% \label{eq:stochastic1}
% \end{equation}
% where $\sigma(x_j) =  \frac{1}{1 + \exp \left[\left( - x_j \right) \right]}$ is the Sigmoid type activation function.

\begin{equation}
%\boxed{
\begin{aligned}
d x_i(t) &= 
\Big[
x_i (x_i - 0.5)(1 - x_i)
- y_i + 0.27
+ \frac{1}{N}\sum_{j=1}^{N} w_{ij}\,\sigma(x_j)
\Big]\,dt
+ \sigma_x\, dW_i^{x}(t),\\[3pt]
d y_i(t) &= 0.003\,(x_i - y_i)\,dt
+ \sigma_y\, dW_i^{y}(t),
\qquad i = 1,\ldots,N,
\end{aligned}%}
\label{eq:nondegenerate_SDE}
\end{equation}
where $\sigma(x) = \dfrac{1}{1 + \exp(-x)}$ is the sigmoid activation function, 
and $W_i^x(t)$ and $W_i^y(t)$ denote two independent standard Wiener processes 
(i.e., $\mathbb{E}[\,dW_i^x\, dW_i^y\,] = 0$). 
The diffusion matrix $\Sigma$ in Eq.~(7.11) is defined as
\[
\Sigma =
\begin{bmatrix}
\sigma_x & 0 \\[2pt]
0 & \sigma_y
\end{bmatrix},
\qquad
D = \Sigma \Sigma^{\!\top}
= \mathrm{diag}(\sigma_x^2,\, \sigma_y^2),
\]
which is positive definite for $\sigma_x, \sigma_y > 0$. Here, $\Sigma$ represents the diffusion matrix that maps the unit-variance
Brownian increments $dW(t)$ to the stochastic perturbations in the state space,
while D = $\Sigma \Sigma^{\!\top}$
denotes the corresponding diffusion covariance matrix. 
The matrix $D$ characterizes the local variance structure of the stochastic
flow and determines how noise is geometrically distributed in different
directions of the phase space. 
In the context of the Onsager--Machlup functional, 
$D$ plays the role of the metric tensor in the stochastic path space, 
since the instantaneous path deviation 
$\dot{z}_t - f(z_t)$ is weighted by $D^{-1}$:
\begin{equation}
L_{\mathrm{OM}} = 
\int_0^T 
\big\langle 
\dot{z}_t - f(z_t),
D^{-1}(\dot{z}_t - f(z_t))
\big\rangle \, dt.
\end{equation}
Hence, $D$ must be positive definite to ensure that 
the stochastic path probability is well-defined. 
If $D$ is degenerate (rank-deficient), 
the noise does not span the entire tangent space, 
which may lead to numerical instability, 
biased diffusion dynamics, and ill-posedness of the 
Onsager--Machlup action functional.

We simulate and analyze the dynamic behavior of the FitzHugh–Nagumo Neural Network system with a fully connected coupling network. Due to the influence of the network structure on the
group synchronization behavior of the interacting particle system in this model. We selected two
coupling structures: fully connected and small-world network. Due to
The fully connected structure can quickly achieve full
synchronization at a low coupling strength, making it suitable for model simplification mode dynamic analysis. However, the connection structure of real neural systems is often not
completely uniform, but rather closer to the characteristics of a small-world network. That is to say, most connections are local connections that coexist with a small number of long-range connections.

We further utilize the Watts-Strogatz small-world network to generate the coupling matrix and find that the system under this structure exhibits more complex dynamic behaviors, including partial synchronization and group fluctuations. These phenomena are closer to the dynamic patterns observed in the actual neural system. The key factors of synchronization are: network structure, coupling strength, and consistency of local dynamics. From the perspective of physiological rationality, the small-world coupling structure is more suitable as the basis for modeling the dynamics of neural networks. In contrast, the fully connected structure is only used as an idealized reference model in theoretical analysis. And finally we choose coupling intensity $K =14$ as system parameter. Based on the synchronization phenomenon we observed, we will now demonstrate the dynamical properties of a single neuron model and a two-dimensional model by manifold reduction.

\subsubsection{Visulization}

First, we analyze the dynamic properties of the deterministic system in the case of a single node ($N=1$), and the corresponding deterministic system becomes:
\begin{equation}
\begin{aligned}
d x &= \left[  x (x - 0.5)(1 - x) - y + I + a  \frac{1}{1 + \exp \left( - x \right)} \right] dt, \\
d y &= 0.003 \left( x - y \right) dt,
\end{aligned}
\label{eq:stochastic}
\end{equation}
where $a$ represents the bifurcation parameter of this system. We studied the dynamic properties of the system under different parameter a conditions and found that when a=0.25, the system exhibits a bistable structure (a stable point and an oscillation limit cycle) as shown in Figure \ref{fig:phase}.

\begin{figure}[h]
    \centering
    \includegraphics[width=0.6\textwidth]{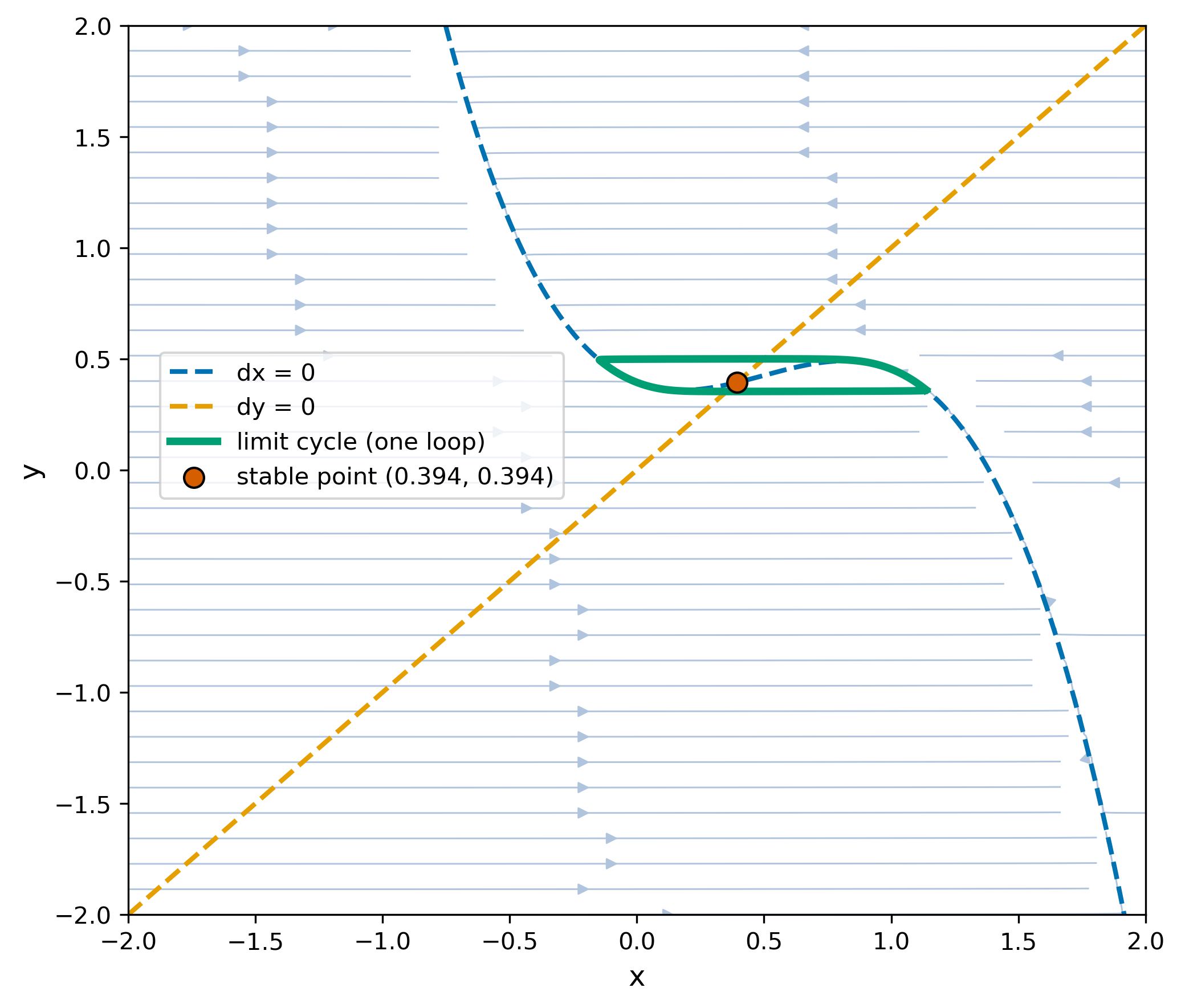}
    \caption{Vector Field with Limit Cycle and Stable Point: the main plot shows the stable point and limit cycle of the system when the bifurcation parameter $a=0.25$. }
    \label{fig:phase}
\end{figure}

\begin{figure}[h]
    \centering
    \includegraphics[width=0.5\textwidth]{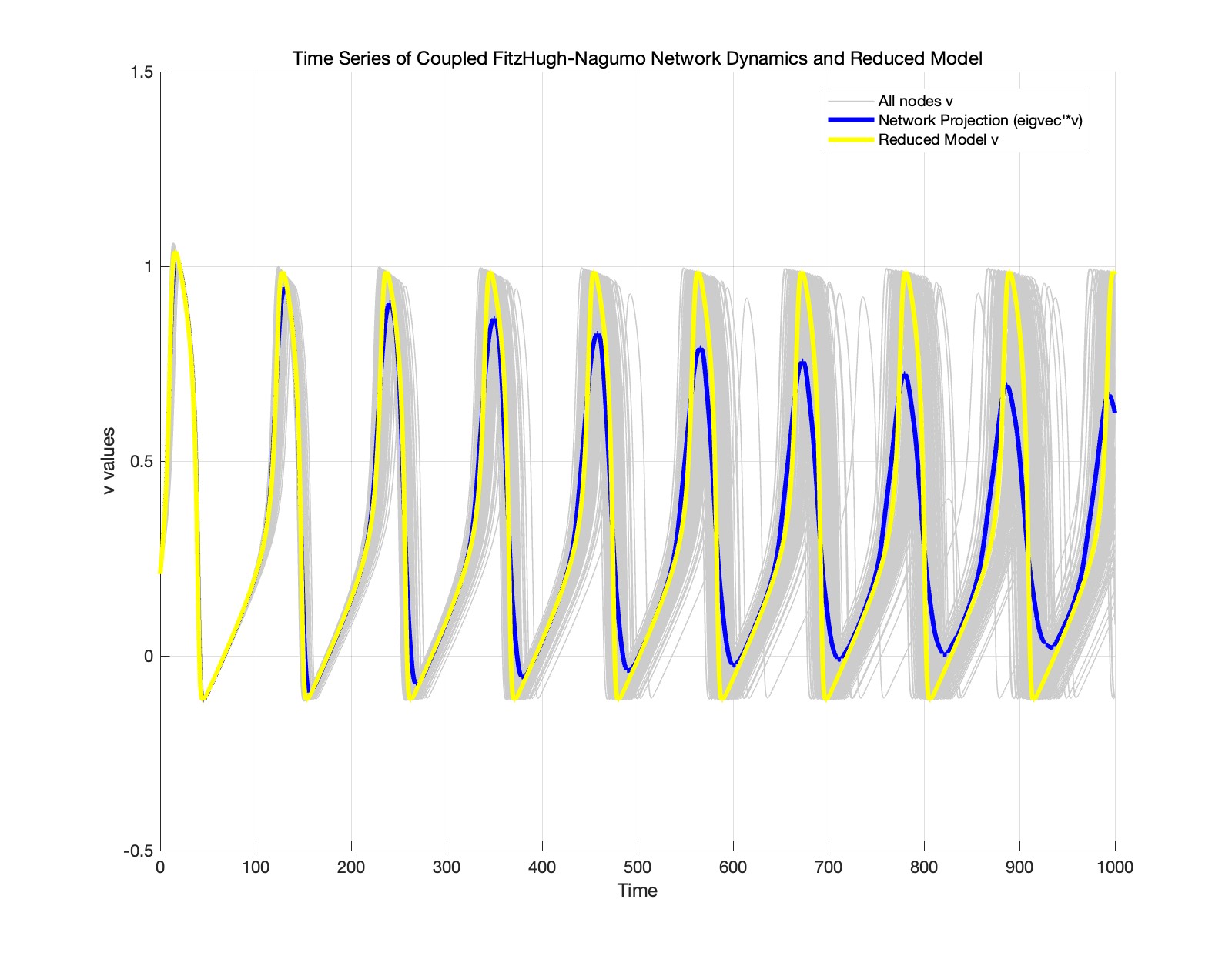}
    \caption{Time series comparison of the FitzHugh-Nagumo neuronal network dynamics and its reduced-order model.}
    \label{fig:timeseries}
\end{figure}

%%%%%%%%%5 这里要放同步化和reduced model 的对比图，说明用reduced model 可以很好的近似原来的粒子系统。

We intend to use manifold reduction to reduce the dimensionality of the original system \cite{QinLin2023}. Figure \ref{fig:timeseries} shows the rationality of this approximation.
The light gray curves represent the membrane potential of all $N=100$ individual neurons in the coupled network with the Watts-Strogatz small world topology. The thick blue line shows the projection of the high-dimensional network state onto the leading eigenvector of the coupling matrix, capturing the dominant collective mode. The yellow thick line corresponds to the solution of the reduced two-dimensional FitzHugh-Nagumo model derived via spectral reduction. The close agreement between the projected network dynamics and the reduced model validates the spectral reduction approach in capturing the macroscopic behavior of the large-scale neuronal network.

In the absence of noise, the coupled neuron system \eqref{eq:stochastic} exhibits rich dynamical behaviors driven by intrinsic nonlinear interactions. Specifically, we consider a node of $N=200$ FitzHugh–Nagumo-like units with excitatory coupling through a sigmoid activation function. The following observations are made based on the simulation results: Figure~\ref{fig:x1_time} shows the temporal evolution of the membrane potential of a single neuron $x_1(t)$, demonstrating clear periodic oscillations. This is a typical manifestation of limit-cycle behavior. Figure~\ref{fig:phase_xy} presents the phase trajectory of the same neuron in the $(x_1, y_1)$ plane, forming a closed, loop-shaped orbit. Relaxation-type oscillations are characterized by slow and fast phases due to the inherent time-scale separation. Figure~\ref{fig:xi_all} displays the time series of all neuron states $\{x_i(t)\}_{i=1}^N$. Initially, the neurons evolve asynchronously. However, as time progresses, the trajectories gradually synchronize onto a common periodic orbit, showing complete phase and frequency synchronization. Figure~\ref{fig:std_x} depicts the standard deviation $\mathrm{std}(x_i(t))$ across the network, which gradually decays to zero. This quantitatively confirms that the system converges to a fully synchronized state along the emergent resonance limit cycle.

These results clearly demonstrate the spontaneous emergence of collective oscillations and synchronization in a population of non-identical neuron units with local coupling, even in the absence of stochastic perturbations.

\begin{figure}[htbp]
    \centering
    \begin{subfigure}[b]{0.48\textwidth}
        \centering
        \includegraphics[width=\linewidth]{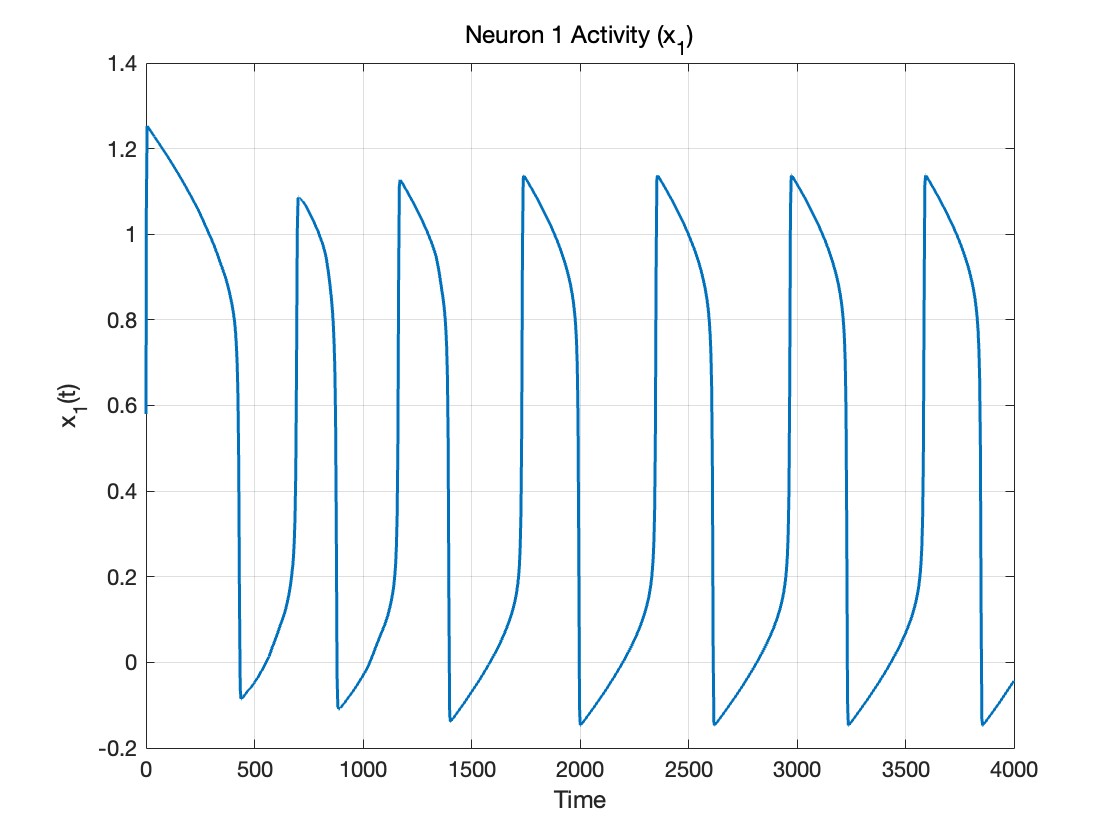}
        \caption{Periodic oscillation orbit of a single neuron. $x_1(t)$}
        \label{fig:x1_time}
    \end{subfigure}
    \hfill
    \begin{subfigure}[b]{0.48\textwidth}
        \centering
        \includegraphics[width=\linewidth]{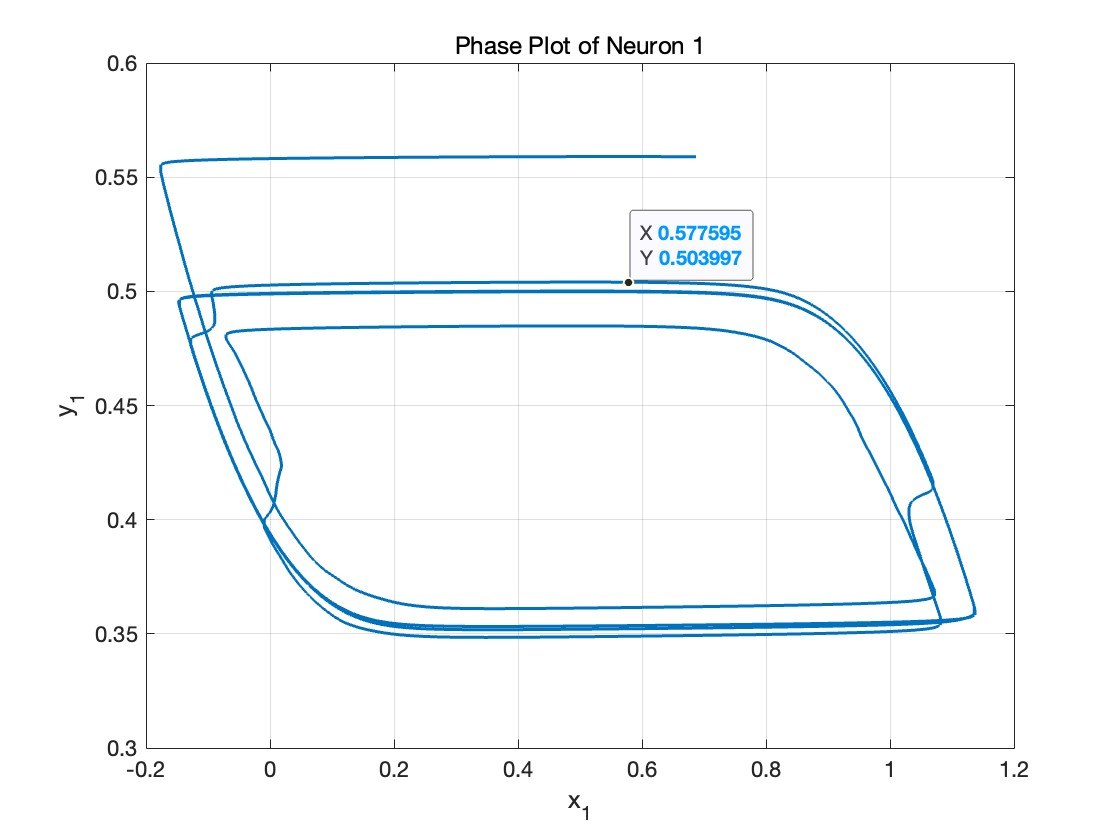}
        \caption{Phase diagram trace showing limit cycles of $x_1$ vs. $y_1$ .}
        \label{fig:phase_xy}
    \end{subfigure}
    
    \vspace{0.3cm}
    
    \begin{subfigure}[b]{0.48\textwidth}
        \centering
        \includegraphics[width=\linewidth]{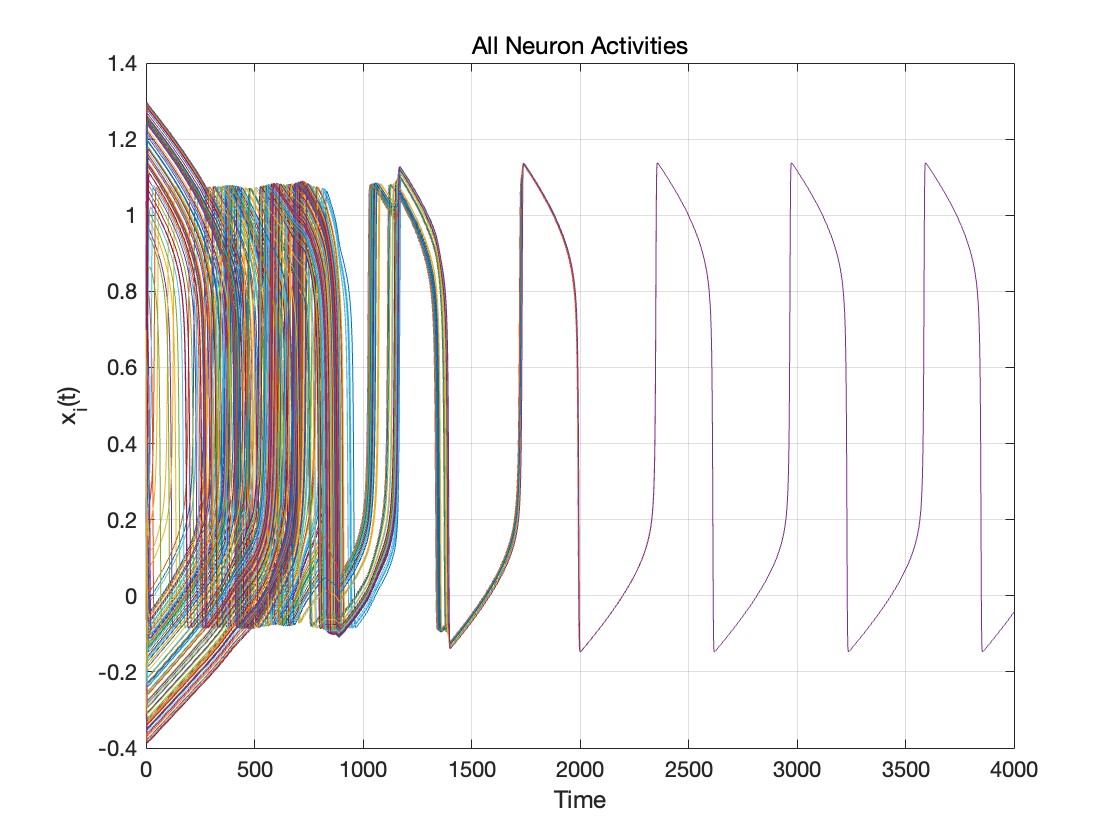}
        \caption{All neuron variables $x_i(t)$ show the final synchronization.}
        \label{fig:xi_all}
    \end{subfigure}
    \hfill
    \begin{subfigure}[b]{0.48\textwidth}
        \centering
        \includegraphics[width=\linewidth]{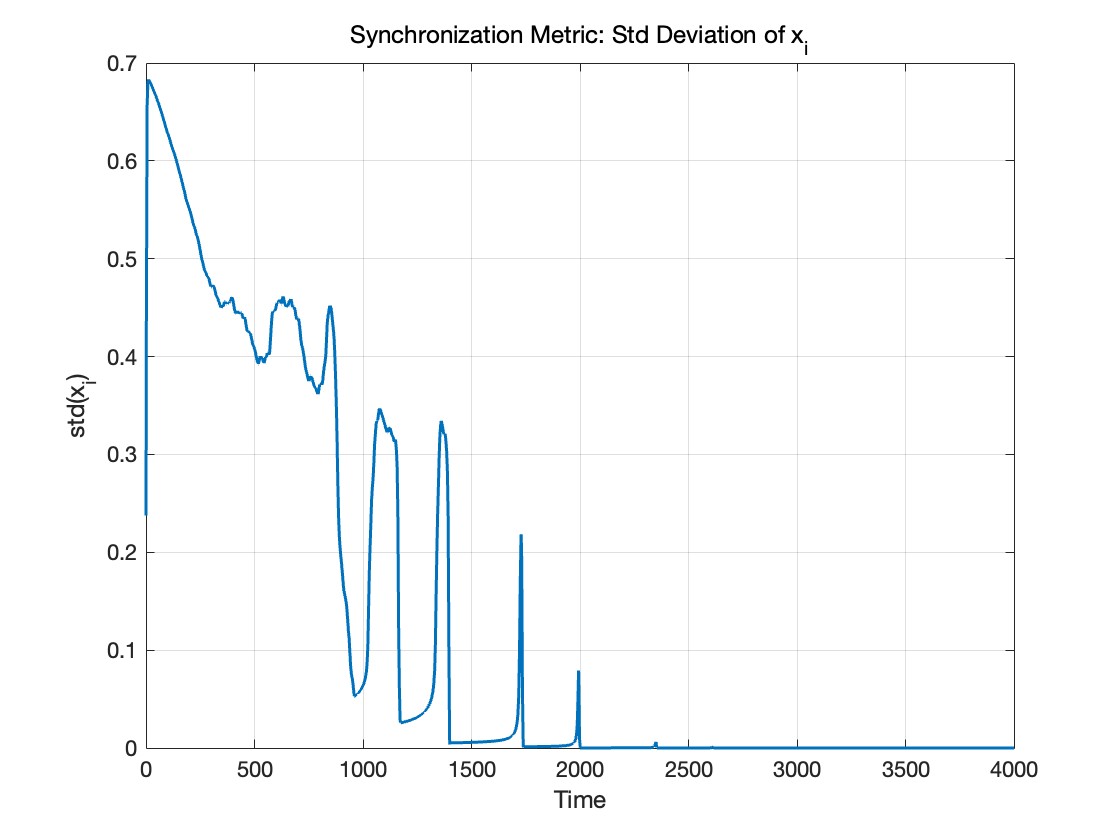}
        \caption{The standard deviation $\mathrm{std}(x_i(t))$ indicates synchronous convergence.}
        \label{fig:std_x}
    \end{subfigure}
    
    \caption{Schematic diagram of the synchronization and limit cycle behavior of a neuron system under noise-free conditions. From the trajectory of a single neuron to the evolution of the entire network, stable synchronous oscillations are exhibited.}
    \label{fig:synchronization}
\end{figure}

To investigate the steady-state behavior of the coupled neuron system, we simulate the deterministic dynamics over a long time horizon ($T = 5000$ with step size $\Delta t = 0.1$) and extract the last 1000 time steps as the empirical steady-state window. The histograms of all $x_i$ and $y_i$ values during this window are shown in Figure~\ref{fig:steady_hist}. As seen in Figure~\ref{fig:steady_hist}, both $x$ and $y$ exhibit sharply peaked unimodal distributions, indicating that each neuron tends to oscillate within a narrow amplitude range in the long-time limit. This behavior aligns with the expected dynamics of relaxation oscillators constrained on a stable limit cycle. Interestingly, the joint distribution plot in Figure~\ref{fig:joint_scatter} reveals that the $(x, y)$ pairs lie approximately along a narrow curve, closely resembling the underlying limit cycle trajectory in phase space. This implies that, despite the dimensionality of the full system, the long-time behavior is effectively constrained to a low-dimensional manifold, further confirming the emergence of synchronized collective oscillations.

Overall, these results illustrate that the entire neural population converges to a synchronized periodic attractor, characterized by a coherent distribution of phase-locked oscillations in the absence of stochastic perturbations.

\begin{figure}[h]
    \centering
    \includegraphics[width=0.8\textwidth]{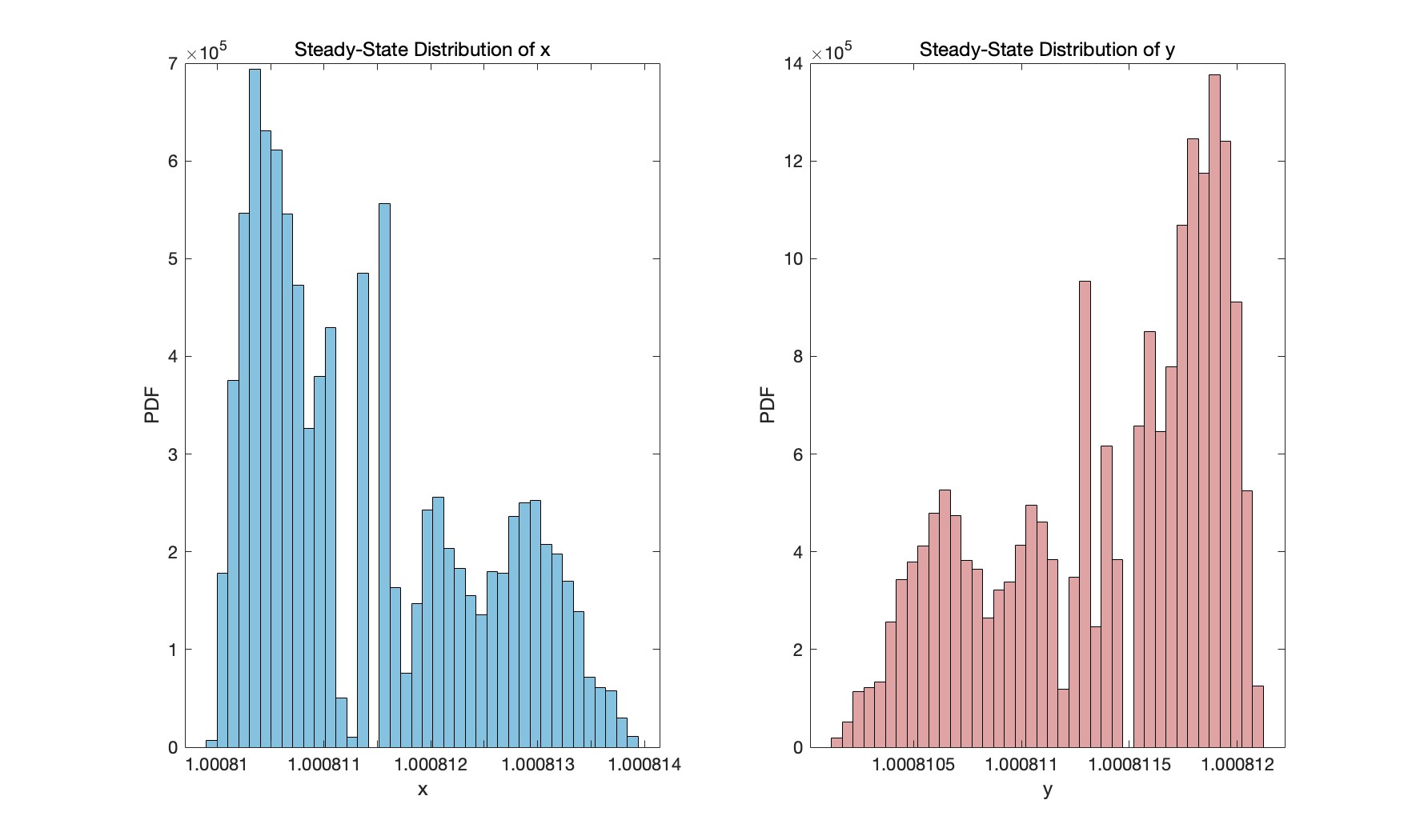}
    \caption{Phase diagram: the main plot on the left shows the stable point and limit cycle of the system when the bifurcation parameter $a=0.25$, and the sub-plot on the right is a local enlarged view of the intersection of the zero slope line, that is the stable point of the system. }
    \label{fig:steady_hist}
\end{figure}

\begin{figure}[h]
    \centering
    \includegraphics[width=0.8\textwidth]{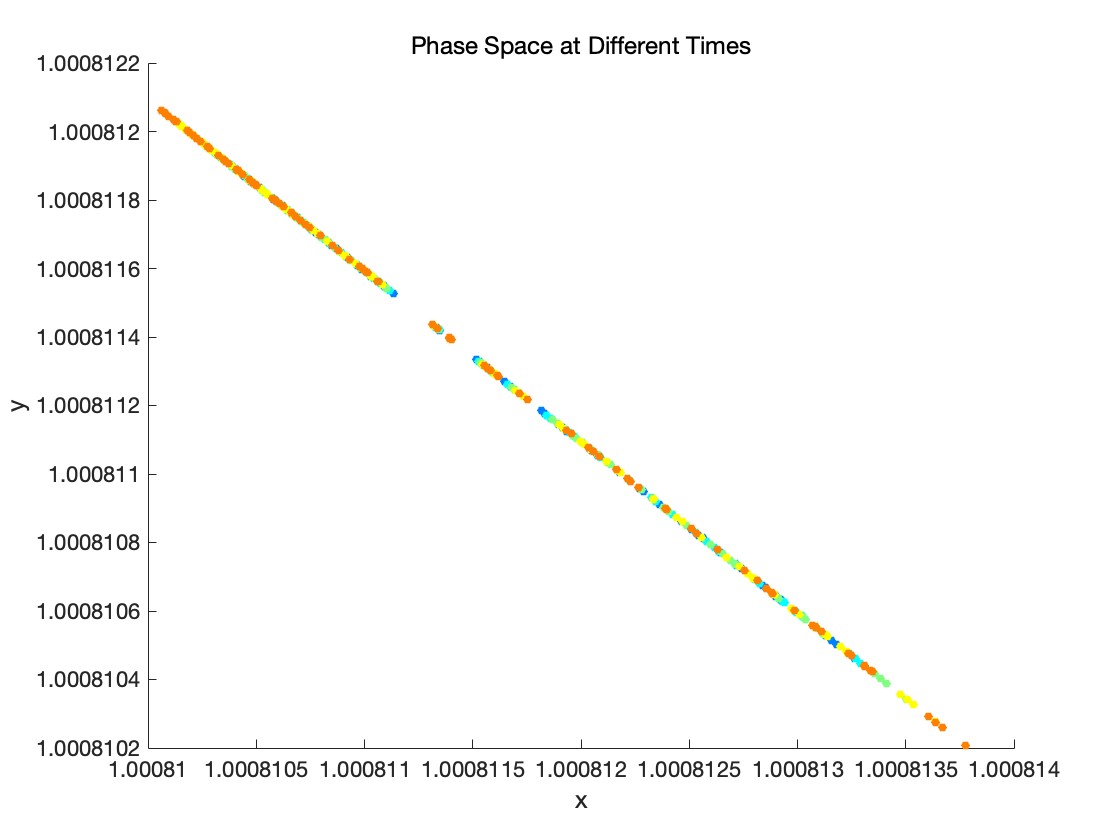}
    \caption{Phase diagram: the main plot on the left shows the stable point and limit cycle of the system when the bifurcation parameter $a=0.25$, and the sub-plot on the right is a local enlarged view of the intersection of the zero slope line, that is the stable point of the system. }
    \label{fig:joint_scatter}
\end{figure}

\section{Innovations}
This paper makes the following innovative contributions to the field of stochastic dynamical systems analysis and modeling:

\begin{enumerate}
    \item \textbf{First Systematic Exploration and Application of OM Functional in Slow-Fast Systems} \\
    Traditionally, the Onsager--Machlup (OM) functional has been extensively used in path sampling, rare event simulation, and optimal control problems for stochastic differential equations (SDEs) with a single time scale. However, in slow-fast systems---stochastic systems with significant time scale separation between variables---the direct embedding of the OM functional poses challenges due to path integral measures, drift-diffusion terms, and the rate functional’s multiscale decomposition. To date, there is little literature systematically applying the OM functional to the analysis and control of such slow-fast stochastic dynamics. This paper pioneers the application of the OM functional in slow-fast systems, establishing a modeling framework that reveals the multiscale geometric structure of slow-fast system path sampling and rare events.
    
    \item \textbf{Numerical Algorithms and Structured Analysis of the OM Functional under Multiscale Decomposition} \\
    For path sampling in slow-fast systems, this paper designs and implements a multiscale decomposition algorithm, splitting the OM functional into coupled fast and slow variable components. By leveraging numerical multiscale techniques (e.g., local averaging, projection operators), we effectively compute the optimal path or rare event probability distribution while preserving the geometric structure of the OM functional. This approach significantly reduces computational complexity in high-dimensional slow-fast systems.
    
    \item \textbf{Integration of Theory and Applications, Advancing Rare Event and Optimal Control Studies} \\
    This work not only proposes novel theoretical formulations of the OM functional under multiscale decomposition but also validates the methodology through concrete case studies---such as high-dimensional chemical reaction networks, molecular dynamics models, or slow-fast processes in biological systems. This innovative attempt provides new perspectives and tools for tackling rare event path sampling, optimal control, and uncertainty quantification in complex slow-fast systems.
\end{enumerate}

\section{Conclusion}\label{disscussion}
In this paper, we study the optimal control problem of McKean-Vlasov stochastic differential equation corresponding to the mean field limit equation of a class of randomly interacting particle systems. Based on the existing research foundation, the correspondence of the maximum possible migration orbit between the random interacting particle system and the McKean-Vlasov random dynamic system is derived from the sense of the Onsager Machlup action functional.

In the research process, we first introduce the stochastic interacting particle system and the mean field limit equation in detail, and consider the stochastic optimal control problem of McKean-Vlasov stochastic differential equation under the action of independent and equally distributed Brown noise. Secondly, we consider the high dimensional stochastic dynamical system under the action of additive Brown noise
Onsager-Machlup action functional, and Onsager-Machlup action functional for a special class of McKean-Vlasov stochastic differential equations. Based on the Pontryagin maximum principle, the expression of optimal control function is obtained from the perspective of optimal control and stochastic optimal control respectively. The main work of this paper is in Section 5.4, which deduces the existence and uniqueness theorem of the solution of stochastic optimal control problem of McKean-Vlasov stochastic differential equation, and further deduces that under the framework of optimal control problem, The correspondence between the optimal governing equation $F_N(\theta)$ for a stochastic interacting particle system and the solution of the optimal governing equation $F(\theta^{*})$ for a McKean-Vlasov stochastic dynamic system, Thus, we can indirectly explore the correspondence of the maximum possible migration orbit of the stochastic interacting particle system.

% In summary, although the Pontryagin maximum principle based on optimal control theory can effectively solve the maximum possible transfer orbit of high-dimensional stochastic dynamical systems, we still raise some challenges and further thinking. Although in numerically applied particles, we calculate the maximum possible migration orbit of a three-dimensional nutrient-phytoplankton-zooplankton (NPZ) system, this is not a high-dimensional stochastic dynamical system in the true sense.
In this paper, the correspondence of the maximum possible transition pathway of the stochastic interacting particle system under the mean field is based on the mean field approximation theorem of the stochastic interacting particle system. Because of the different dimensions of the space where the maximum possible transition pathway is located, it is difficult to approach the transition pathway  directly. This paper takes into account the fact that the maximum possible transition pathway  is controlled by the control of the system (i.e. Brownian noise), because we indirectly prove the correspondence of the maximum possible transition pathway of a stochastic interacting particle system and its mean field limit equation in the sense of the optimal control solution. It lays a foundation for studying the maximum possible transition pathway of stochastic interacting particle systems.

We study the correspondence between the most probable transition pathways of a class of stochatic interacting particle systems and their mean field limit equations (McKean-Vlasov stochastic differential equations) under Brownian noise. The dimensionality of a stochatic interacting particle system is usually very high (particle number $N \rightarrow{\infty}$), and it is difficult to directly solve the maximum possible transition pathway of a particle system either from the variational principle or from the perspective of optimal control. The optimal control problem of the most probable transition pathways for the average field limit stochastic dynamical system is established, and the correspondence between the core equation (Hamiltonian maximum condition) in the Pontryagin maximum principle is obtained under the optimal control principle, which is helpful for the study of the most probable transition pathways' properties of the stochatic interacting particle system with high dimensions.

%%%%%%%%
\section{Future Work}
In this paper, based on the Onsager Machlup functional theory, we prove the correspondence between the maximum possible migration orbit of the random interacting particle system and its mean field limit equation from the perspective of optimal control. However, there are still some shortcomings: (i) Since the optimal control problem is based on the Onagaser-Machlup action functional theory, the OnAgaser-Machlup action functional for McKean-Vlasov stochastic differential equation is limited to a special class of drift function $f$. The large deviation theory can perfectly avoid this limitation, so in the subsequent research, we can use the large deviation principle to extend this correspondence theorem to stochastic dynamical systems with a more general drift function $f$. (ii) In the study of this paper, we only provided theoretical derivation results, but did not design numerical experiments, so we could not intuitively see the correspondence between the maximum possible migration orbit of the randomly interacting particle system and its average field limit system. (iii) In the current research on random interacting particle systems, Gaussian Brown noise is the most commonly introduced noise. Due to the relatively good orbital properties, the effect will be small when studying the dynamical behavior of the system, and it may even bring about good structure. At present, however, random noise can be simulated by other random processes with discontinuous orbits, such as non-Gaussian L\'{e}vy processes and Poisson jump processes. Therefore, in future work, we will consider more complex noise types and explore the maximum possible migration orbit effect of discontinuous noise on randomly interacting particle systems.

\section*{Acknowledgements}
%The work was done partially while the author was visiting the Center for Mathematical Sciences, Huazhong University of Science and Technology, Wuhan, China. 

This work was partially supported by the Guangdong Provincial Key Laboratory of Mathematical and Neural Dynamical Systems (2024B1212010004) and the National Natural Science Foundation of China (No. 12141107).
The author would like to thank Dr. Zibo Wang for stimulating discussions and constructive suggestions of this work.

% \section*{Data Availability}
% The data that support the findings of this study are openly available in GitHub.\\
% \url{https://github.com/Cecilia-ChenJY/PMP-for-Optimal-Control}

\bibliographystyle{unsrt}
 %\bibliographystyle{plain}
% \biboptions{square,numbers,sort&compress}
\bibliography{references}
% \begin{appendices}
% \section{Pontryagin's Maximum Principle of stochastic Maier-Stein System.}\label{appA.1}
% \end{appendices}
\end{document}